\documentclass[11pt,reqno]{amsart} 
\usepackage[a4paper, margin=3cm]{geometry}

\usepackage{graphicx}
\graphicspath{ {./images/} }
\usepackage{lipsum}% Just for this example
\usepackage{subcaption}

\usepackage{xcolor}

\usepackage{amsthm}

\usepackage{bbm}
\usepackage{mathrsfs}
\usepackage{dynkin-diagrams}

\usepackage[utf8]{inputenc}
\usepackage{xypic}

\usepackage{tikz-cd}

\usepackage{enumitem}

\usepackage{amssymb}
\usepackage{float}
\usepackage{amsbsy}
\usepackage[all]{xy}\usepackage{xypic}
\usepackage{braket}
\usepackage[colorlinks = true,citecolor=black]{hyperref}
\usepackage{amsmath}

\newcommand\myeq{\mathrel{\stackrel{\makebox[0pt]{\mbox{\normalfont\tiny loc}}}{=}}}

\hypersetup{
     colorlinks=true,
     linkcolor=blue,
     filecolor=blue,
     citecolor = blue,      
     urlcolor=cyan,
     }

\makeatletter
\newcommand\opteq[1]{\mathrel{\mathpalette\opt@eq{#1}}}
\newcommand{\opt@eq}[2]{%
  \begingroup
  \sbox\z@{$#1#2$}%
  \sbox\tw@{\resizebox{!}{.5\ht\z@}{$\m@th#1($}}%
  \nonscript\hskip-\wd\tw@
  \mkern1mu
  \raisebox{-.35\ht\z@}[0pt][0pt]{\resizebox{!}{.5\ht\z@}{$\m@th#1($}}%
  \mkern-1mu
  {#2}%
  \mkern-1mu
  \raisebox{-.35\ht\z@}[0pt][0pt]{\resizebox{!}{.5\ht\z@}{$\m@th#1)$}}%
  \mkern1mu
  \nonscript\hskip-\wd\tw@
  \endgroup
}
\makeatother

\newcommand{\geoq}{\opteq{\geq}}
\newtheorem{theorem}{Theorem}[section]
\newtheorem*{theorem*}{Theorem}
\newtheorem{theorem-non}{Theorem}
\newtheorem{lemma-non}{Lemma}

\theoremstyle{definition} 
\newtheorem{thm}{Theorem}

\theoremstyle{definition}

\newtheorem{conjecture-non}{Conjecture}

\newtheorem{corollary-non}{Corollary}
\newtheorem{proposition}[theorem]{Proposition}
\newtheorem{lemma}[theorem]{Lemma}
\newtheorem*{lemma*}{Lemma}

\newtheorem*{conjecture*}{Conjecture}

\theoremstyle{definition}
\newtheorem{definition}[theorem]{Definition}

\theoremstyle{remark}
\newtheorem{remark}[theorem]{Remark}

\DeclareMathOperator{\rank}{rank}

\numberwithin{equation}{section}

\begin{document}
\title[Eder M. Correa]{DHYM connections on higher rank holomorphic vector bundles over ${\mathbbm{P}}(T_{{\mathbbm{P}^{2}}})$}

\author{Eder M. Correa}

%\thanks{Eder M. Correa was supported by PRPq-UFMG grant 27764*32}

\address{{IMECC-Unicamp, Departamento de Matem\'{a}tica. Rua S\'{e}rgio Buarque de Holanda 651, Cidade Universit\'{a}ria Zeferino Vaz. 13083-859, Campinas-SP, Brazil}}
\address{E-mail: {\rm ederc@unicamp.br}}

\begin{abstract} 
We construct the first explicit non-trivial example of deformed Hermitian Yang-Mills (dHYM) connection on a higher rank slope-unstable holomorphic vector bundle over a Fano threefold. Additionally, we provide a sufficient algebraic condition in terms of central charges for the existence of dHYM connections on Whitney sum of holomorphic line bundles over rational homogeneous varieties. As a consequence, we obtain several new examples of dHYM connections on higher rank holomorphic vector bundles.
\end{abstract}

\maketitle

\hypersetup{linkcolor=black}
%\tableofcontents

\hypersetup{linkcolor=black}

\section{Introduction}
Let $(X,\omega)$ be a compact connected K\"{a}hler manifold, such that $\dim_{\mathbbm{C}}(X) = n$, and $[\psi] \in H^{1,1}(X,\mathbbm{R})$. The {\textit{deformed Hermitian Yang-Mills}} (dHYM) {\textit{equation}} asks for a canonical representative $\chi \in [\psi]$ satisfying 
\begin{equation}
\label{dHYMeq}
{\rm{Im}}\big ( \omega + \sqrt{-1}\chi\big)^{n} = \tan(\hat{\Theta}) {\rm{Re}}\big ( \omega + \sqrt{-1}\chi\big)^{n},
\end{equation}
such that $\hat{\Theta} = {\rm{Arg}} \int_{X}\frac{(\omega + \sqrt{-1}\psi)^{n}}{n!} \ ({\rm{mod}} \ 2\pi)$. This equation was originally derived in the physics literature for the case of line bundles, e.g. \cite{marino2000nonlinear}, \cite{leung2000special}, from a mathematical perspective, the equation arises through SYZ mirror symmetry as the dual to the special Lagrangian equation on a Calabi–Yau manifold. The analytical study of the dHYM equation was initiated by Jacob–Yau \cite{JacobYau2017} and further explored in a series of works \cite{Collins2018deformed, Collins2020, Collins2021moment, Collins2020stability}, see also \cite{Pingali2019}, \cite{Chen2021j}, and \cite{chu2021nakai}. As it can be shown, Eq. (\ref{dHYMeq}) has an alternative (equivalent) formulation in terms of the notion of {\textit{Lagrangian phase}} \cite{Collins2020}, more precisely, Eq. (\ref{dHYMeq}) is equivalent to the fully nonlinear elliptic equation
\begin{equation}
\label{dHYMeq2v}
\Theta_{\omega}(\chi) := \sum_{j = 1}^{n}\arctan(\lambda_{j}) = \hat{\Theta} \ ({\rm{mod}} \ 2\pi),
\end{equation}
where $\lambda_{1},\ldots,\lambda_{n}$ are the eigenvalues of $\omega^{-1} \circ \chi$. In this last equation $\Theta_{\omega}(\chi)$ is called the Lagrangian phase of $\chi$ with respect to $\omega$. In \cite{correa2023deformed}, by describing explicitly the Lagrangian phase of every closed invariant real $(1,1)$-form in terms of Lie theory, the author proved that the dHYM equation always admits a solution if $(X,\omega)$ is a rational homogeneous variety. In higher rank, we have the following generalization suggested by Collins–Yau \cite{collins2018moment}: Given a holomorphic vector bundle ${\bf{E}} \to (X,\omega)$, we say that a Hermitian metric ${\bf{h}}$ on ${\bf{E}}$ solves the dHYM equation if the curvature $F_{\nabla}$ of the associated Chern connection $\nabla$ satisfies
\begin{equation}
\label{dHYMhigher}
{\rm{Im}}\Big ( {\rm{e}}^{-\sqrt{-1} \hat{\Theta}({\bf{E}})}\Big (\omega \otimes \mathbbm{1}_{{\bf{E}}} - \frac{1}{2\pi}F_{\nabla} \Big )^{n} \Big ) = 0,
\end{equation}
such that 
\begin{equation}
\label{phaseanglehigherbundle}
\hat{\Theta}({\bf{E}}) = {\rm{Arg}}\int_{X}{\rm{tr}}\Big (\omega \otimes \mathbbm{1}_{{\bf{E}}} - \frac{1}{2\pi}F_{\nabla}\Big)^{n} \ (\text{mod} \ 2\pi).
\end{equation}
In the above setting, we call $\nabla$ satisfying Eq. (\ref{dHYMhigher}) a {\textit{dHYM connection}}. There are by now in the literature many important results concerned with dHYM connections on holomorphic line bundles, and few results are known in the higher rank setting. In \cite{dervan2020z}, strong mathematical justification that the higher rank deformed Hermitian Yang-Mills equation suggested by Collins-Yau is indeed the appropriate equation was provided. Moreover, among other results, in \cite{dervan2020z} the authors introduced the notion of Z{\textit{-critical connection}} and proved that, in the large volume limit, a sufficiently smooth holomorphic vector bundle \cite{leung1997einstein} admits a Z-critical connection if and only if it is {\textit{asymptotically} Z{\textit{-stable}}. In this paper, we study Eq. (\ref{dHYMhigher}) on holomorphic vector bundles over rational homogeneous varieties equipped with some invariant K\"{a}hler metric. A rational homogeneous variety can be described as a quotient $X_{P} = G^{\mathbbm{C}}/P$, where $G^{\mathbbm{C}}$ is a semisimple complex algebraic group with Lie algebra $\mathfrak{g}^{\mathbbm{C}} = {\rm{Lie}}(G^{\mathbbm{C}})$, and $P$ is a parabolic Lie subgroup (Borel-Remmert \cite{BorelRemmert}). Regarding $G^{\mathbbm{C}}$ as a complex analytic space, without loss of generality, we may assume that $G^{\mathbbm{C}}$ is a connected simply connected complex simple Lie group. Fixed a compact real form $G \subset G^{\mathbbm{C}}$, one can consider $X_{P} = G/G \cap P$ as a $G$-space. In order to state our first result, let us recall some terminology. Fixed a K\"{a}hler class $\xi \in \mathcal{K}(X_{P})$, the {\textit{slope}} $\mu_{\xi}({\bf{E}})$ of a holomorphic vector bundle ${\bf{E}} \to X_{P}$, with respect to $\xi$, is defined by
\begin{equation}
\mu_{\xi}({\bf{E}}):= \frac{\int_{X_{P}}c_{1}({\bf{E}}) \wedge \xi^{n-1}}{\rank({\bf{E}})}.
\end{equation}
From above, a holomorphic vector bundle ${\bf{E}} \to X_{P}$ is said to be {\textit{slope-(semi)stable}} if 
\begin{equation}
\mu_{\xi}({\bf{E}}) \geoq \mu_{\xi}(\mathcal{F}),
\end{equation}
for every subsheaf $0 \neq \mathcal{F} \varsubsetneq {\bf{E}}$. Further, we say that ${\bf{E}}$ is {\textit{slope-polystable}} if it is isomorphic to a direct sum of stable vector bundles of the same slope, and we say that ${\bf{E}}$ is {\textit{slope-unstable}} if it is not slope-semistable. Considering the rational homogeneous Fano threefold ${\mathbbm{P}}(T_{{\mathbbm{P}^{2}}})$, motivated by \cite[Conjecture 1.3]{MR4479717}, \cite{dervan2020z}, and by the increasing interest in the existence of Bridgeland stability conditions on Fano threefolds \cite{MR3121850}, \cite{MR3370123}, \cite{MR3573975}, \cite{MR3454685, koseki2018stability}, \cite{MR3994590, MR3908763}, \cite{MR4127165}, we prove the following theorem.
\begin{thm}
\label{TheoA}
Fixed the unique ${\rm{SU}}(3)$-invariant K\"{a}hler metric $\omega_{0} \in  c_{1}({\mathbbm{P}}(T_{{\mathbbm{P}^{2}}}))$, then there exist holomorphic vector bundles ${\bf{E}}_{j} \to {\mathbbm{P}}(T_{{\mathbbm{P}^{2}}})$, $j=1,2,3$, satisfying the following:
\begin{enumerate}
\item $\rank({\bf{E}}_{j}) > 1$, $j = 1,2,3$;
\item ${\bf{E}}_{1}$ admits a Hermitian structure ${\bf{h}}$ with associated Chern connection $\nabla$ satisfying
\begin{equation}
\begin{cases}\sqrt{-1}\Lambda_{\omega_{0}}(F_{\nabla}) = c\mathbbm{1}_{{\bf{E}}_{1}}, \\ {\rm{Im}}\Big ({\rm{e}}^{-\sqrt{-1}\hat{\Theta}({\bf{E}}_{1})}\Big (\omega_{0} \otimes \mathbbm{1}_{{\bf{E}}_{1}} - \frac{1}{2\pi}F_{\nabla}\Big )^{3}\Big ) =  0,\end{cases}
\end{equation}
such that $c = \frac{\pi}{8} \mu_{[\omega_{0}]}({\bf{E}}_{1})$ and $\hat{\Theta}({\bf{E}}_{1}) = {\rm{Arg}}\int_{{\mathbbm{P}}(T_{{\mathbbm{P}^{2}}})}{\rm{tr}}\Big (\omega_{0} \otimes \mathbbm{1}_{{\bf{E}}_{1}} - \frac{1}{2\pi}F_{\nabla}\Big)^{3} \ (\text{mod} \ 2\pi)$.
\item ${\bf{E}}_{2}$ admits a Hermitian structure ${\bf{h}}$ with associated Chern connection $\nabla$ satisfying 
\begin{equation}
\begin{cases}\sqrt{-1}\Lambda_{\omega_{0}}(F_{\nabla}) = c\mathbbm{1}_{{\bf{E}}_{2}}, \\ {\rm{Im}}\Big ({\rm{e}}^{-\sqrt{-1}\hat{\Theta}({\bf{E}}_{2})}\Big (\omega_{0} \otimes \mathbbm{1}_{{\bf{E}}_{2}} - \frac{1}{2\pi}F_{\nabla}\Big )^{3}\Big ) \neq  0,\end{cases}
\end{equation}
for $c = \frac{\pi}{8} \mu_{[\omega_{0}]}({\bf{E}}_{2})$ and for all $\hat{\Theta}({\bf{E}}_{2}) \in \mathbbm{R}$.
\item ${\bf{E}}_{3}$ admits a Hermitian structure ${\bf{h}}$ with associated Chern connection $\nabla$ satisfying
\begin{equation}
\begin{cases} \sqrt{-1}\Lambda_{\omega_{0}}(F_{\nabla}) \neq c\mathbbm{1}_{{\bf{E}}_{3}}, \\ {\rm{Im}}\Big ({\rm{e}}^{-\sqrt{-1}\hat{\Theta}({\bf{E}}_{3})}\Big (\omega_{0} \otimes \mathbbm{1}_{{\bf{E}}_{3}} - \frac{1}{2\pi}F_{\nabla}\Big )^{3}\Big ) = 0,\end{cases}
\end{equation}
for all $c \in \mathbbm{R}$, such that $\hat{\Theta}({\bf{E}}_{3}) = = {\rm{Arg}}\int_{{\mathbbm{P}}(T_{{\mathbbm{P}^{2}}})}{\rm{tr}}\Big (\omega_{0} \otimes \mathbbm{1}_{{\bf{E}}_{3}} - \frac{1}{2\pi}F_{\nabla}\Big)^{3} \ (\text{mod} \ 2\pi)$.
\end{enumerate}
In particular, ${\bf{E}}_{1}$ and ${\bf{E}}_{2}$ are slope-polystable, and ${\bf{E}}_{3}$ is slope-unstable.
\end{thm}
In order to prove the above result, we construct explicitly the Hermitian connections satisfying the aforementioned properties. In particular, from item (4) of the above theorem, we obtain the first explicit non-trivial example of dHYM connection on a higher rank slope-unstable holomorphic vector bundle over a Fano threefold. The first non-trivial solutions to the dHYM equation and Z-critical equation in a higher rank (semistable) holomorphic vector bundle were provided in \cite{dervan2020z} for $X = {\mathbbm{P}^{2}}$, results related to the existence of dHYM connections on slope-unstable holomorphic vector bundles are not known so far. It is worth mentioning that our construction is independent of the results provided in \cite{dervan2020z}. 

Given a rational homogeneous variety $X_{P}$, let $\omega_{0}$ be a $G$-invariant K\"{a}hler metric on $X_{P}$. In this setting, for every Hermitian holomorphic vector bundle $({\bf{E}},{\bf{h}}) \to (X_{P},\omega_{0})$, denoting by $\nabla$ the associated Chern connection, and considering the representative ${\rm{ch}}_{k}({\bf{E}},\nabla)$ of the cohomology class defined by the $k$-th Chern character ${\rm{ch}}_{k}({\bf{E}}) \in H^{2k}(X_{P},\mathbbm{C})$, we define the {\textit{central charge}}
\begin{equation}
\label{Centralcharge}
Z_{[\omega_{0}]}({\bf{E}}) = -\int_{X_{P}}{\rm{e}}^{-\sqrt{-1}[\omega_{0}]}{\rm{ch}}({\bf{E}}) = -\sum_{j=0}^{n}\frac{(-\sqrt{-1})^{j}}{j!}\int_{X_{P}}\omega_{0}^{j} \wedge {\rm{ch}}_{n-j}({\bf{E}},\nabla).
\end{equation}
As mentioned in \cite{Collins2020}, the complex number $Z_{[\omega_{0}]}(-)$ resembles various notions of central charge appearing in the study of stability conditions in several physical and mathematical theories, see for instance \cite{douglas2001a, douglas2001b,douglas2002c} and \cite{thomas2002stability}. In the particular case that ${\bf{E}}$ is a line bundle, it is conjectured \cite[Conjecture 1.3]{MR4479717} that the existence of a dHYM connection $\nabla$ on ${\bf{E}}$ should be equivalent to the Bridgeland stability \cite{Bridgeland2007stability} of ${\bf{E}}$, for an introduction to Bridgeland stability we suggest \cite{Macri2017lectures}. We have by now several results related to algebraic conditions involving certain central charges and intersection numbers with the existence of the dHYM connections on holomorphic line bundles, e.g. \cite{Collins2020}, \cite{Chen2021j}, \cite{datar2021numerical}, \cite{chu2021nakai}. Motivated by these results and by the ideas introduced in \cite{correa2023deformed}, we derive a sufficient algebraic condition in terms of the central charge given in Eq. (\ref{Centralcharge}) for the existence of dHYM connections on Whitney sum of holomorphic line bundles over rational homogeneous varieties. More precisely, we prove the following.
\begin{thm}
\label{TheoB}
Given ${\bf{E}}, {\bf{F}} \in {\rm{Pic}}(X_{P})$, if 
\begin{equation}
{\rm{Im}}\bigg ( \frac{Z_{[\omega_{0}]}({\bf{E}})}{Z_{[\omega_{0}]}({\bf{F}})} \bigg ) = 0,
\end{equation}
then ${\bf{E}} \oplus {\bf{F}}$ admits a dHYM instanton.
\end{thm}

The above result allows us to construct examples of higher rank dHYM connections from elements of ${\rm{Pic}}(X_{P})$. Notice that, since ${\rm{ch}}({\bf{E}}) = {\rm{e}}^{c_{1}({\bf{E}})}$, $\forall {\bf{E}} \in {\rm{Pic}}(X_{P})$, it follows that
\begin{equation}
Z_{[\omega_{0}]}({\bf{E}}) = -\frac{(-\sqrt{-1})^{n}}{n!}\int_{X_{P}}\big([\omega_{0}] + \sqrt{-1}c_{1}({\bf{E}})\big)^{n}, \ \ \forall {\bf{E}} \in {\rm{Pic}}(X_{P}).
\end{equation}
Following \cite{correa2023deformed}, the integral which appears on the right-hand side of the above equation can be explicitly computed in concrete cases using tools from Lie theory. As an application of this last fact, and the ideas introduced in Theorem \ref{TheoA} and Theorem \ref{TheoB}, we explore the relationship between dHYM connections and Z-critical connections \cite{dervan2020z}. More precisely, in the setting of Eq. (\ref{Centralcharge}), considering the ${\rm{End}}({\bf{E}})$-valued $(n,n)$-form
\begin{equation}
Z_{\omega_{0}}({\bf{E}},\nabla) := -\sum_{j=0}^{n}\frac{(-\sqrt{-1})^{j}}{j!}\omega_{0}^{j} \wedge \widehat{{\rm{ch}}_{n-j}({\bf{E}},\nabla)},
\end{equation}
such that\footnote{Notice that ${\rm{tr}}(\widehat{{\rm{ch}}({\bf{E}},\nabla)})$ is a representative for the cohomology class defined by the Chern character ${\rm{ch}}({\bf{E}})$.} 
\begin{equation}
\widehat{{\rm{ch}}({\bf{E}},\nabla)} := \exp \bigg ( \frac{\sqrt{-1}}{2\pi}F_{\nabla}\bigg ) = \widehat{{\rm{ch}}_{0}({\bf{E}},\nabla)} + \widehat{{\rm{ch}}_{1}({\bf{E}},\nabla)} + \cdots + \widehat{{\rm{ch}}_{n}({\bf{E}},\nabla)}, 
\end{equation}
we say that $\nabla$ is a $Z$-critical connection \cite{dervan2020z} if
\begin{equation}
{\rm{Im}}\Big ({\rm{e}}^{-\sqrt{-1}\varphi({\bf{E}})}Z_{\omega_{0}}({\bf{E}},\nabla) \Big) = 0,
\end{equation}
such that $\varphi({\bf{E}}) = {\rm{Arg}}\big(Z_{[\omega_{0}]}({\bf{E}})\big )\ ({\rm{mod}} \ 2\pi)$. In this setting, we prove the following.
\begin{thm}
\label{theoC}
Under the hypotheses of Theorem \ref{TheoA}, for every integer $r > 1$, there exists a Hermitian holomorphic vector bundle $({\bf{E}},{\bf{h}}) \to ({\mathbbm{P}}(T_{{\mathbbm{P}^{2}}}),\omega_{0})$, such that $\rank({\bf{E}}) = r$, and the following hold:
\begin{enumerate}
\item ${\bf{E}}$ is slope-unstable and $h^{0}({\mathbbm{P}}(T_{{\mathbbm{P}^{2}}}),{\rm{End}}({\bf{E}})) > 1$;
\item considering $\varphi({\bf{E}}) = {\rm{Arg}}\big(Z_{[\omega_{0}]}({\bf{E}})\big )\ ({\rm{mod}} \ 2\pi)$, we have
\begin{equation}
\label{Zcritical}
{\rm{Im}}\Big ({\rm{e}}^{-\sqrt{-1}\varphi({\bf{E}})}Z_{\omega_{0}}({\bf{E}},\nabla) \Big) = 0,
\end{equation}
\end{enumerate}
where $\nabla$ is the Chern connection associated with ${\bf{h}}$.
\end{thm}
It is worth to point out that Eq. (\ref{Zcritical}) is equivalent to Eq. (\ref{dHYMhigher}). Therefore, from the above theorem we also obtain several new examples of dHYM connections on higher rank holomorphic vector bundles. The proof of the above result provides a constructive method to obtain examples of Z-critical connections on slope-unstable not simple holomorphic vector bundles over ${\mathbbm{P}}(T_{{\mathbbm{P}^{2}}})$. In particular, Theorem \ref{theoC} provides the first example in the literature of a Z-critical connection on a slope-unstable not simple holomorphic vector bundle over a Fano threefold.

{\bf{Organization of the paper.}} In Section 2, we review some basic generalities on flag varieties. In Section 3, we prove Theorem \ref{TheoA}, the proof is constructive and can be summarized in Lemma \ref{L1} (Section \ref{dHYM_polystable}), Lemma \ref{L2} (Section \ref{dHYM_polystable}), and Lemma \ref{L3} (Section \ref{Unstableconstruction}). In Section \ref{Proof_theo_B_C}, we prove Theorem \ref{TheoB} and Theorem \ref{theoC}.

{\bf{Acknowledgments.}}  The author would like to thank Professor Lino Grama for very helpful conversations. E. M. Correa is partially supported by FAEPEX/Unicamp grant 2528/22 and by S\~{a}o Paulo Research Foundation FAPESP grant 2022/10429-3.

\section{Generalities on flag varieties}\label{generalities}
In this section, we review some basic generalities about flag varieties. For more details on the subject presented in this section, we suggest \cite{Akhiezer}, \cite{Flagvarieties}, \cite{HumphreysLAG}, \cite{BorelRemmert}.
\subsection{The Picard group of flag varieties}
\label{subsec3.1}
Let $G^{\mathbbm{C}}$ be a connected, simply connected, and complex Lie group with simple Lie algebra $\mathfrak{g}^{\mathbbm{C}}$. By fixing a Cartan subalgebra $\mathfrak{h}$ and a simple root system $\Delta \subset \mathfrak{h}^{\ast}$, we have a decomposition of $\mathfrak{g}^{\mathbbm{C}}$ given by
\begin{center}
$\mathfrak{g}^{\mathbbm{C}} = \mathfrak{n}^{-} \oplus \mathfrak{h} \oplus \mathfrak{n}^{+}$, 
\end{center}
where $\mathfrak{n}^{-} = \sum_{\alpha \in \Phi^{-}}\mathfrak{g}_{\alpha}$ and $\mathfrak{n}^{+} = \sum_{\alpha \in \Phi^{+}}\mathfrak{g}_{\alpha}$, here we denote by $\Phi = \Phi^{+} \cup \Phi^{-}$ the root system associated with the simple root system $\Delta \subset \mathfrak{h}^{\ast}$. Let us denote by $\kappa$ the Cartan-Killing form of $\mathfrak{g}^{\mathbbm{C}}$. From this, for every  $\alpha \in \Phi^{+}$, we have $h_{\alpha} \in \mathfrak{h}$, such  that $\alpha = \kappa(\cdot,h_{\alpha})$, and we can choose $x_{\alpha} \in \mathfrak{g}_{\alpha}$ and $y_{-\alpha} \in \mathfrak{g}_{-\alpha}$, such that $[x_{\alpha},y_{-\alpha}] = h_{\alpha}$. From these data, we can define a Borel subalgebra\footnote{A maximal solvable subalgebra of $\mathfrak{g}^{\mathbbm{C}}$.} by setting $\mathfrak{b} = \mathfrak{h} \oplus \mathfrak{n}^{+}$. 

\begin{remark}
In the above setting, $\forall \phi \in \mathfrak{h}^{\ast}$, we also denote $\langle \phi, \alpha \rangle = \phi(h_{\alpha})$, $\forall \alpha \in \Phi^{+}$.
\end{remark}

Now we consider the following result (see for instance \cite{Flagvarieties}, \cite{HumphreysLAG}):
\begin{theorem}
Any two Borel subgroups are conjugate.
\end{theorem}
From the result above, given a Borel subgroup $B \subset G^{\mathbbm{C}}$, up to conjugation, we can always suppose that $B = \exp(\mathfrak{b})$. In this setting, given a parabolic Lie subgroup\footnote{A Lie subgroup which contains some Borel subgroup.} $P \subset G^{\mathbbm{C}}$, without loss of generality, we can suppose that
\begin{center}
$P  = P_{I}$, \ for some \ $I \subset \Delta$,
\end{center}
where $P_{I} \subset G^{\mathbbm{C}}$ is the parabolic subgroup which integrates the Lie subalgebra 
\begin{center}

$\mathfrak{p}_{I} = \mathfrak{n}^{+} \oplus \mathfrak{h} \oplus \mathfrak{n}(I)^{-}$, \ with \ $\mathfrak{n}(I)^{-} = \displaystyle \sum_{\alpha \in \langle I \rangle^{-}} \mathfrak{g}_{\alpha}$. 

\end{center}
By definition, we have that $P_{I} = N_{G^{\mathbbm{C}}}(\mathfrak{p}_{I})$, where $N_{G^{\mathbbm{C}}}(\mathfrak{p}_{I})$ is the normalizer in  $G^{\mathbbm{C}}$ of $\mathfrak{p}_{I} \subset \mathfrak{g}^{\mathbbm{C}}$, see for instance \cite[\S 3.1]{Akhiezer}. In what follows, it will be useful for us to consider the following basic chain of Lie subgroups

\begin{center}

$T^{\mathbbm{C}} \subset B \subset P \subset G^{\mathbbm{C}}$.

\end{center}
For each element in the aforementioned chain of Lie subgroups we have the following characterization: 

\begin{itemize}

\item $T^{\mathbbm{C}} = \exp(\mathfrak{h})$;  \ \ (complex torus)

\item $B = N^{+}T^{\mathbbm{C}}$, where $N^{+} = \exp(\mathfrak{n}^{+})$; \ \ (Borel subgroup)

\item $P = P_{I} = N_{G^{\mathbbm{C}}}(\mathfrak{p}_{I})$, for some $I \subset \Delta \subset \mathfrak{h}^{\ast}$. \ \ (parabolic subgroup)

\end{itemize}
Now let us recall some basic facts about the representation theory of $\mathfrak{g}^{\mathbbm{C}}$, a detailed exposition on the subject can be found in \cite{Humphreys}. For every $\alpha \in \Phi$, we set 
$$\alpha^{\vee} := \frac{2}{\langle \alpha, \alpha \rangle}\alpha.$$ 
The fundamental weights $\{\varpi_{\alpha} \ | \ \alpha \in \Delta\} \subset \mathfrak{h}^{\ast}$ of $(\mathfrak{g}^{\mathbbm{C}},\mathfrak{h})$ are defined by requiring that $\langle \varpi_{\alpha}, \beta^{\vee} \rangle= \delta_{\alpha \beta}$, $\forall \alpha, \beta \in \Delta$. We denote by 
$$\Lambda^{+} = \bigoplus_{\alpha \in \Delta}\mathbbm{Z}_{\geq 0}\varpi_{\alpha},$$ 
the set of integral dominant weights of $\mathfrak{g}^{\mathbbm{C}}$. Let $V$ be an arbitrary finite dimensional $\mathfrak{g}^{\mathbbm{C}}$-module. By considering its weight space decomposition
\begin{center}
$\displaystyle{V = \bigoplus_{\mu \in \Phi(V)}V_{\mu}},$ \ \ \ \ 
\end{center}
such that $V_{\mu} = \{v \in V \ | \ h \cdot v = \mu(h)v, \ \forall h \in \mathfrak{h}\} \neq \{0\}$, $\forall \mu \in \Phi(V) \subset \mathfrak{h}^{\ast}$, we have the following definition.

\begin{definition}
A highest weight vector (of weight $\lambda$) in a $\mathfrak{g}^{\mathbbm{C}}$-module $V$ is a non-zero vector $v_{\lambda}^{+} \in V_{\lambda}$, such that 
\begin{center}
$x \cdot v_{\lambda}^{+} = 0$, \ \ \ \ \ ($\forall x \in \mathfrak{n}^{+}$).
\end{center}
A weight $\lambda \in \Phi(V)$ associated with a highest weight vector is called highest weight of $V$.
\end{definition}

From above, we consider the following standard results.

\begin{theorem}
Every finite dimensional irreducible $\mathfrak{g}^{\mathbbm{C}}$-module $V$ admits a highest weight vector $v_{\lambda}^{+}$. Moreover, $v_{\lambda}^{+}$ is the unique highest weight vector of $V$, up to non-zero scalar multiples. 
\end{theorem}

\begin{theorem}
Let $V$ and $W$ be finite dimensional irreducible $\mathfrak{g}^{\mathbbm{C}}$-modules with highest weight $\lambda \in \mathfrak{h}^{\ast}$. Then, $V$ and $W$ are isomorphic.
\end{theorem}
\begin{remark}
We will denote by $V(\lambda)$ a finite dimensional irreducible $\mathfrak{g}^{\mathbbm{C}}$-module with highest weight $\lambda \in \mathfrak{h}^{\ast}$.
\end{remark}
\begin{theorem} In the above setting, the following hold:
\begin{itemize}
\item[(1)] If $V$ is a finite dimensional irreducible $\mathfrak{g}^{\mathbbm{C}}$-module with highest weight $\lambda \in \mathfrak{h}^{\ast}$, then $\lambda \in \Lambda^{+}$.

\item[(2)] If $\lambda \in \Lambda^{+}$, then there exists a finite dimensional irreducible $\mathfrak{g}^{\mathbbm{C}}$-module $V$, such that $V = V(\lambda)$. 
\end{itemize}
\end{theorem}
From the above theorem, it follows that the map $\lambda \mapsto V(\lambda)$ induces an one-to-one correspondence between $\Lambda^{+}$ and the isomorphism classes of finite dimensional irreducible $\mathfrak{g}^{\mathbbm{C}}$-modules.

\begin{remark} In what follows, it will be useful also to consider the following facts:
\begin{enumerate}
\item[(i)] For all $\lambda \in \Lambda^{+}$, we have $V(\lambda) = \mathfrak{U}(\mathfrak{g}^{\mathbbm{C}}) \cdot v_{\lambda}^{+}$, where $\mathfrak{U}(\mathfrak{g}^{\mathbbm{C}})$ is the universal enveloping algebra of $\mathfrak{g}^{\mathbbm{C}}$;
\item[(ii)] The fundamental representations are defined by $V(\varpi_{\alpha})$, $\alpha \in \Delta$; 
\item[(iii)] For all $\lambda \in \Lambda^{+}$, we have the following equivalence of induced irreducible representations
\begin{center}
$\varrho \colon G^{\mathbbm{C}} \to {\rm{GL}}(V(\lambda))$ \ $\Longleftrightarrow$ \ $\varrho_{\ast} \colon \mathfrak{g}^{\mathbbm{C}} \to \mathfrak{gl}(V(\lambda))$,
\end{center}
such that $\varrho(\exp(x)) = \exp(\varrho_{\ast}x)$, $\forall x \in \mathfrak{g}^{\mathbbm{C}}$, notice that $G^{\mathbbm{C}} = \langle \exp(\mathfrak{g}^{\mathbbm{C}}) \rangle$.
\end{enumerate}
\end{remark}
Given a representation $\varrho \colon G^{\mathbbm{C}} \to {\rm{GL}}(V(\lambda))$, for the sake of simplicity, we shall denote $\varrho(g)v = gv$, for all $g \in G^{\mathbbm{C}}$, and all $v \in V(\lambda)$. Let $G \subset G^{\mathbbm{C}}$ be a compact real form for $G^{\mathbbm{C}}$. Given a complex flag variety $X_{P} = G^{\mathbbm{C}}/P$, regarding $X_{P}$ as a homogeneous $G$-space, that is, $X_{P} = G/G\cap P$, the following theorem allows us to describe all $G$-invariant K\"{a}hler structures on $X_{P}$ through elements of representation theory.
\begin{theorem}[Azad-Biswas, \cite{AZAD}]
\label{AZADBISWAS}
Let $\omega \in \Omega^{1,1}(X_{P})^{G}$ be a closed invariant real $(1,1)$-form, then we have

\begin{center}

$\pi^{\ast}\omega = \sqrt{-1}\partial \overline{\partial}\varphi$,

\end{center}
where $\pi \colon G^{\mathbbm{C}} \to X_{P}$ is the natural projection, and $\varphi \colon G^{\mathbbm{C}} \to \mathbbm{R}$ is given by 
\begin{center}
$\varphi(g) = \displaystyle \sum_{\alpha \in \Delta \backslash I}c_{\alpha}\log \big (||gv_{\varpi_{\alpha}}^{+}|| \big )$, \ \ \ \ $(\forall g \in G^\mathbbm{C})$
\end{center}
with $c_{\alpha} \in \mathbbm{R}$, $\forall \alpha \in \Delta \backslash I$. Conversely, every function $\varphi$ as above defines a closed invariant real $(1,1)$-form $\omega_{\varphi} \in \Omega^{1,1}(X_{P})^{G}$. Moreover, $\omega_{\varphi}$ defines a $G$-invariant K\"{a}hler form on $X_{P}$ if and only if $c_{\alpha} > 0$,  $\forall \alpha \in \Delta \backslash I$.
\end{theorem}

\begin{remark}
\label{innerproduct}
It is worth pointing out that the norm $|| \cdot ||$ considered in the above theorem is a norm induced from some fixed $G$-invariant inner product $\langle \cdot, \cdot \rangle_{\alpha}$ on $V(\varpi_{\alpha})$, $\forall \alpha \in \Delta \backslash I$. 
\end{remark}

\begin{remark}
An important consequence of Theorem \ref{AZADBISWAS} is that it allows us to describe the local K\"{a}hler potential for any homogeneous K\"{a}hler metric in a quite concrete way, for some examples of explicit computations, we suggest \cite{CorreaGrama}, \cite{Correa}.
\end{remark}

By means of the above theorem we can describe the unique $G$-invariant representative of each integral class in $H^{2}(X_{P},\mathbbm{Z})$. In fact, consider the holomorphic $P$-principal bundle $P \hookrightarrow G^{\mathbbm{C}} \to X_{P}$. By choosing a trivializing open covering $X_{P} = \bigcup_{i \in J}U_{i}$, in terms of $\check{C}$ech cocycles we can write 
\begin{center}
$G^{\mathbbm{C}} = \Big \{(U_{i})_{i \in J}, \psi_{ij} \colon U_{i} \cap U_{j} \to P \Big \}$.
\end{center}
Given $\varpi_{\alpha} \in \Lambda^{+}$, we consider the induced character $\vartheta_{\varpi_{\alpha}} \in {\text{Hom}}(T^{\mathbbm{C}},\mathbbm{C}^{\times})$, such that $({\rm{d}}\vartheta_{\varpi_{\alpha}})_{e} = \varpi_{\alpha}$. Since $P = P_{I}$, we have the decomposition  
\begin{equation}
\label{parabolicdecomposition}
P_{I} = \big[P_{I},P_{I} \big]T(\Delta \backslash I)^{\mathbbm{C}}, \ \  {\text{such that }} \ \ T(\Delta \backslash I)^{\mathbbm{C}} = \exp \Big \{ \displaystyle \sum_{\alpha \in  \Delta \backslash I}a_{\alpha}t_{\alpha} \ \Big | \ a_{\alpha} \in \mathbbm{C} \Big \},
\end{equation}
such that $t_{\alpha} := \frac{2}{\langle \alpha,\alpha \rangle}h_{\alpha}, \forall \alpha \in \Delta \backslash I$, e.g. \cite[\S 3]{Akhiezer}, so we can consider the extension $ \vartheta_{\varpi_{\alpha}} \in {\text{Hom}}(P,\mathbbm{C}^{\times})$. From the homomorphism $\vartheta_{\varpi_{\alpha}} \colon P \to \mathbbm{C}^{\times}$ one can equip $\mathbbm{C}$ with a structure of $P$-space, such that $pz = \vartheta_{\varpi_{\alpha}}(p)^{-1}z$, $\forall p \in P$, and $\forall z \in \mathbbm{C}$. Denoting by $\mathbbm{C}_{-\varpi_{\alpha}}$ this $P$-space, we can form an associated holomorphic line bundle $\mathscr{O}_{\alpha}(1) = G^{\mathbbm{C}} \times_{P}\mathbbm{C}_{-\varpi_{\alpha}}$, which can be described in terms of $\check{C}$ech cocycles by
\begin{equation}
\label{linecocycle}
\mathscr{O}_{\alpha}(1) = \Big \{(U_{i})_{i \in J},\vartheta_{\varpi_{\alpha}}^{-1} \circ \psi_{i j} \colon U_{i} \cap U_{j} \to \mathbbm{C}^{\times} \Big \},
\end{equation}
that is, $\mathscr{O}_{\alpha}(1) = \{g_{ij}\} \in \check{H}^{1}(X_{P},\mathcal{O}_{X_{P}}^{\ast})$, such that $g_{ij} = \vartheta_{\varpi_{\alpha}}^{-1} \circ \psi_{i j}$, $\forall i,j \in J$. 
\begin{remark}
\label{parabolicdec}
Given a parabolic Lie subgroup $P \subset G^{\mathbbm{C}}$, such that $P = P_{I}$, for some $I \subset \Delta$, the decomposition Eq. (\ref{parabolicdecomposition}) shows us that ${\text{Hom}}(P,\mathbbm{C}^{\times}) = {\text{Hom}}(T(\Delta \backslash I)^{\mathbbm{C}},\mathbbm{C}^{\times})$. Therefore, given $\vartheta \in {\text{Hom}}(T^{\mathbbm{C}},\mathbbm{C}^{\times})$, we consider $\vartheta \in {\text{Hom}}(P,\mathbbm{C}^{\times})$ as being the trivial extension of $\vartheta|_{T(\Delta \backslash I)^{\mathbbm{C}}}$ to $P$. In particular, if we take $\varpi_{\alpha} \in \Lambda^{+}$, such that $\alpha \in I$, since $\vartheta_{\varpi_{\alpha}}|_{T(\Delta \backslash I)^{\mathbbm{C}}}$ is trivial, it follows that $\mathscr{O}_{\alpha}(1)$ is trivial.
\end{remark}

\begin{remark}
Throughout this paper we shall use the following notation
\begin{equation}
\mathscr{O}_{\alpha}(k) := \mathscr{O}_{\alpha}(1)^{\otimes k},
\end{equation}
for every $k \in \mathbbm{Z}$ and every $\alpha \in \Delta \backslash I$. 
\end{remark}

Given $\mathscr{O}_{\alpha}(1) \in {\text{Pic}}(X_{P})$, such that $\alpha \in \Delta \backslash I$, as described above, if we consider an open covering $X_{P} = \bigcup_{i \in J} U_{i}$ which trivializes both $P \hookrightarrow G^{\mathbbm{C}} \to X_{P}$ and $ \mathscr{O}_{\alpha}(1) \to X_{P}$, by taking a collection of local sections $(s_{i})_{i \in J}$, such that $s_{i} \colon U_{i} \to G^{\mathbbm{C}}$, we can define $q_{i} \colon U_{i} \to \mathbbm{R}^{+}$, such that 
\begin{equation}
\label{functionshermitian}
q_{i} := \frac{1}{||s_{i}v_{\varpi_{\alpha}}^{+}||^{2}},
\end{equation}
for every $i \in J$. Since $s_{j} = s_{i}\psi_{ij}$ on $U_{i} \cap U_{j} \neq \emptyset$, and $pv_{\varpi_{\alpha}}^{+} = \vartheta_{\varpi_{\alpha}}(p)v_{\varpi_{\alpha}}^{+}$, for every $p \in P$, and every $\alpha \in \Delta \backslash I$, the collection of functions $(q_{i})_{i \in J}$ satisfy $q_{j} = |\vartheta_{\varpi_{\alpha}}^{-1} \circ \psi_{ij}|^{2}q_{i}$ on $U_{i} \cap U_{j} \neq \emptyset$. Hence, we obtain a collection of functions $(q_{i})_{i \in J}$ which satisfies on the overlaps $U_{i} \cap U_{j} \neq \emptyset$ the following relation
\begin{equation}
\label{collectionofequ}
q_{j} = |g_{ij}|^{2}q_{i},
\end{equation}
such that $g_{ij} = \vartheta_{\varpi_{\alpha}}^{-1} \circ \psi_{i j}$, $\forall i,j \in J$. From this, we can define a Hermitian structure ${\bf{h}}$ on $\mathscr{O}_{\alpha}(1)$ by taking on each trivialization $f_{i} \colon \mathscr{O}_{\alpha}(1)|_{U_{i}} \to U_{i} \times \mathbbm{C}$ the metric defined by
\begin{equation}
\label{hermitian}
{\bf{h}}(f_{i}^{-1}(x,v),f_{i}^{-1}(x,w)) = q_{i}(x) v\overline{w},
\end{equation}
for every $(x,v),(x,w) \in U_{i} \times \mathbbm{C}$. The Hermitian metric above induces a Chern connection $\nabla \myeq {\rm{d}} + \partial \log {\bf{h}}$ with curvature $F_{\nabla}$ satisfying (locally)
\begin{equation}
\displaystyle \frac{\sqrt{-1}}{2\pi}F_{\nabla} \myeq \frac{\sqrt{-1}}{2\pi} \partial \overline{\partial}\log \Big ( \big | \big | s_{i}v_{\varpi_{\alpha}}^{+}\big | \big |^{2} \Big).
\end{equation}
Therefore, by considering the closed $G$-invariant $(1,1)$-form ${\bf{\Omega}}_{\alpha} \in \Omega^{1,1}(X_{P})^{G}$, which satisfies $\pi^{\ast}{\bf{\Omega}}_{\alpha} = \sqrt{-1}\partial \overline{\partial} \varphi_{\varpi_{\alpha}}$, where $\pi \colon G^{\mathbbm{C}} \to G^{\mathbbm{C}} / P = X_{P}$, and $\varphi_{\varpi_{\alpha}}(g) = \frac{1}{2\pi}\log||gv_{\varpi_{\alpha}}^{+}||^{2}$, $\forall g \in G^{\mathbbm{C}}$, we have 
\begin{equation}
{\bf{\Omega}}_{\alpha} |_{U_{i}} = (\pi \circ s_{i})^{\ast}{\bf{\Omega}}_{\alpha} = \frac{\sqrt{-1}}{2\pi}F_{\nabla} \Big |_{U_{i}},
\end{equation}
i.e., $c_{1}(\mathscr{O}_{\alpha}(1)) = [ {\bf{\Omega}}_{\alpha}]$, $\forall \alpha \in \Delta \backslash I$.

\begin{remark}
Given $I \subset \Delta$, we shall denote $\Phi_{I}^{\pm}:= \Phi^{\pm} \backslash \langle I \rangle^{\pm}$, such that $\langle I \rangle^{\pm} = \langle I \rangle \cap \Phi^{\pm}$.
\end{remark}

\begin{remark}
\label{bigcellcosntruction}
In order to perform some local computations we shall consider the open set $U^{-}(P) \subset X_{P}$ defined by the ``opposite" big cell in $X_{P}$. This open set is a distinguished coordinate neighbourhood $U^{-}(P) \subset X_{P}$ of $x_{0} = eP \in X_{P}$ defined as follows
\begin{equation}
\label{bigcell}
 U^{-}(P) =  B^{-}x_{0} = R_{u}(P_{I})^{-}x_{0} \subset X_{P},  
\end{equation}
 where $B^{-} = \exp(\mathfrak{h} \oplus \mathfrak{n}^{-})$, and
 
 \begin{center}
 
 $R_{u}(P_{I})^{-} = \displaystyle \prod_{\alpha \in \Phi_{I}^{+}}N_{\alpha}^{-}$, \ \ (opposite unipotent radical)
 
 \end{center}
with $N_{\alpha}^{-} = \exp(\mathfrak{g}_{-\alpha})$, $\forall \alpha \in \Phi_{I}^{+}$, e.g. \cite[\S 3]{Lakshmibai2},\cite[\S 3.1]{Akhiezer}. It is worth mentioning that the opposite big cell defines a contractible open dense subset in $X_{P}$, thus the restriction of any vector bundle (principal bundle) over this open set is trivial.
\end{remark}

Consider now the following result.

\begin{lemma}
\label{funddynkinline}
Consider $\mathbbm{P}_{\beta}^{1} = \overline{\exp(\mathfrak{g}_{-\beta})x_{0}} \subset X_{P}$, such that $\beta \in \Phi_{I}^{+}$. Then, 
\begin{equation}
\int_{\mathbbm{P}_{\beta}^{1}} {\bf{\Omega}}_{\alpha} = \langle \varpi_{\alpha}, \beta^{\vee}  \rangle, \ \forall \alpha \in \Delta \backslash I.
\end{equation}
\end{lemma}

A proof of the above result can be found in \cite{correa2023deformed}, see also \cite{FultonWoodward} and \cite{AZAD}. From the above lemma and Theorem \ref{AZADBISWAS}, we obtain the following fundamental result.

\begin{proposition}[\cite{correa2023deformed}]
\label{C8S8.2Sub8.2.3P8.2.6}
Let $X_{P}$ be a complex flag variety associated with some parabolic Lie subgroup $P = P_{I}$. Then, we have
\begin{equation}
\label{picardeq}
{\text{Pic}}(X_{P}) = H^{1,1}(X_{P},\mathbbm{Z}) = H^{2}(X_{P},\mathbbm{Z}) = \displaystyle \bigoplus_{\alpha \in \Delta \backslash I}\mathbbm{Z}[{\bf{\Omega}}_{\alpha} ].
\end{equation}
\end{proposition}

%\begin{remark}
%\label{dykinlineroot}
%Combining the above result with Lemma \ref{funddynkinline}, we obtain the following
%\begin{equation}
%[\mathbbm{P}_{\beta}^{1}] = \sum_{\alpha \in \Delta \backslash I}\langle \varpi_{\alpha},\beta^{\vee} \rangle [\mathbbm{P}_{\alpha}^{1}] \ \ \ (\forall \beta \in \Phi_{I}^{+}).
%\end{equation}
%\end{remark}

In the above setting, we consider the weights of $P = P_{I}$ as being  
\begin{center}
$\displaystyle \Lambda_{P} := \bigoplus_{\alpha \in \Delta \backslash I}\mathbbm{Z}\varpi_{\alpha}$. 
\end{center}
From this, the previous result provides $\Lambda_{P} \cong {\rm{Hom}}(P,\mathbbm{C}^{\times}) \cong {\rm{Pic}}(X_{P})$, such that 
\begin{enumerate}
\item$ \displaystyle \lambda = \sum_{\alpha \in \Delta \backslash I}k_{\alpha}\varpi_{\alpha} \mapsto \prod_{\alpha \in \Delta \backslash I} \vartheta_{\varpi_{\alpha}}^{k_{\alpha}} \mapsto \bigotimes_{\alpha \in \Delta \backslash I} \mathscr{O}_{\alpha}(k_{\alpha})$;
\item $ \displaystyle {\bf{E}} \mapsto \vartheta_{{\bf{E}}}: = \prod_{\alpha \in \Delta \backslash I} \vartheta_{\varpi_{\alpha}}^{\langle c_{1}({\bf{E}}),[\mathbbm{P}^{1}_{\alpha}] \rangle} \mapsto \lambda({\bf{E}}) := \sum_{\alpha \in \Delta \backslash I}\langle c_{1}({\bf{E}}),[\mathbbm{P}^{1}_{\alpha}] \rangle\varpi_{\alpha}$.
\end{enumerate}
Thus, $\forall {\bf{E}} \in {\rm{Pic}}(X_{P})$, we have $\lambda({\bf{E}}) \in \Lambda_{P}$. More generally, $\forall \xi \in H^{1,1}(X_{P},\mathbbm{R})$, we have $\lambda (\xi) \in \Lambda_{P}\otimes \mathbbm{R}$, such that
\begin{equation}
\label{weightcohomology}
\lambda(\xi) := \sum_{\alpha \in \Delta \backslash I}\langle \xi,[\mathbbm{P}^{1}_{\alpha}] \rangle\varpi_{\alpha}.
\end{equation}
From above, for every holomorphic vector bundle ${\bf{E}} \to X_{P}$, we define $\lambda({\bf{E}}) \in \Lambda_{P}$, such that 
\begin{equation}
\label{weightholomorphicvec}
\lambda({\bf{E}}) := \sum_{\alpha \in \Delta \backslash I} \langle c_{1}({\bf{E}}),[\mathbbm{P}_{\alpha}^{1}] \rangle \varpi_{\alpha},
\end{equation}
where $c_{1}({\bf{E}}) = c_{1}(\bigwedge^{r}{\bf{E}})$, such that $r = \rank({\bf{E}})$.

\begin{remark}[Harmonic 2-forms on $X_{P}$]Given any $G$-invariant Riemannian metric $g$ on $X_{P}$, let us denote by $\mathscr{H}^{2}(X_{P},g)$ the space of real harmonic 2-forms on $X_{P}$ with respect to $g$, and by $\mathscr{I}_{G}^{1,1}(X_{P})$ the space of closed invariant real $(1,1)$-forms. Combining the result of Proposition \ref{C8S8.2Sub8.2.3P8.2.6} with \cite[Lemma 3.1]{MR528871}, we obtain 
\begin{equation}
\mathscr{I}_{G}^{1,1}(X_{P}) = \mathscr{H}^{2}(X_{P},g). 
\end{equation}
Therefore, the closed $G$-invariant real $(1,1)$-forms described in Theorem \ref{AZADBISWAS} are harmonic with respect to any $G$-invariant Riemannian metric on $X_{P}$.
\end{remark}

\begin{remark}[K\"{a}hler cone of $X_{P}$]It follows from Eq. (\ref{picardeq}) and Theorem \ref{AZADBISWAS} that the K\"{a}hler cone of a complex flag variety $X_{P}$ is given explicitly by
\begin{equation}
\mathcal{K}(X_{P}) = \displaystyle \bigoplus_{\alpha \in \Delta \backslash I} \mathbbm{R}^{+}[ {\bf{\Omega}}_{\alpha}].
\end{equation}
\end{remark}

\begin{remark}[Cone of curves of $X_{P}$] It is worth observing that the cone of curves ${\rm{NE}}(X_{P})$ of a flag variety $X_{P}$ is generated by the rational curves $[\mathbbm{P}_{\alpha}^{1}] \in \pi_{2}(X_{P})$, $\alpha \in \Delta \backslash I$, see for instance \cite[\S 18.3]{Timashev} and references therein.
\end{remark}

\begin{proposition}
\label{eigenvalueatorigin}
Let $X_{P}$ be a flag variety and let $\omega_{0}$ be a $G$-invariant K\"{a}hler metric on $X_{P}$. Then, for every closed $G$-invariant real $(1,1)$-form $\psi$, the eigenvalues of the endomorphism $\omega_{0}^{-1} \circ \psi$ are given by 
\begin{equation}
\label{eigenvalues}
{\bf{q}}_{\beta}(\omega_{0}^{-1} \circ \psi) = \frac{ \langle \lambda([\psi]), \beta^{\vee} \rangle}{\langle \lambda([\omega_{0}]), \beta^{\vee} \rangle}, \ \ \beta \in \Phi_{I}^{+}.
\end{equation}
\end{proposition}

A proof for the above result can be found in \cite{correa2023deformed}.

\begin{remark}
\label{primitivecalc}
In the setting of the last proposition, since $n \psi\wedge \omega_{0}^{n-1} = \Lambda_{\omega_{0}}(\psi)\omega_{0}^{n}$, such that $n=\dim_{\mathbbm{C}}(X_{P})$, and $\Lambda_{\omega_{0}}(\psi)={\rm{tr}}(\omega_{0}^{-1} \circ \psi)$, it follows that
\begin{equation}
\label{contractiongenerators}
\Lambda_{\omega_{0}}({\bf{\Omega}}_{\alpha})=\sum_{\beta \in \Phi_{I}^{+}} \frac{\langle \varpi_{\alpha}, \beta^{\vee} \rangle}{\langle \lambda([\omega_{0}]), \beta^{\vee}\rangle},
\end{equation}
for every $\alpha \in \Delta \backslash I$. In particular, for every ${\bf{E}} \in {\rm{Pic}}(X_{P})$, we have a Hermitian structure ${\bf{h}}$ on ${\bf{E}}$, such that the curvature $F_{\nabla}$ of the Chern connection $\nabla \myeq {\rm{d}} + \partial \log ({\bf{h}})$, satisfies 
\begin{equation}
\frac{\sqrt{-1}}{2\pi} \Lambda_{\omega_{0}}(F_{\nabla}) = \sum_{\beta \in \Phi_{I}^{+} } \frac{\langle \lambda({\bf{E}}), \beta^{\vee} \rangle}{\langle \lambda([\omega_{0}]), \beta^{\vee}\rangle}.
\end{equation}
From this, we have that $\nabla$ is a Hermitian-Yang-Mills (HYM) connection. Notice that 
\begin{equation}
c_{1}({\bf{E}}) = \sum_{\alpha \in \Delta \backslash I}\langle \lambda({\bf{E}}), \alpha^{\vee} \rangle [{\bf{\Omega}}_{\alpha}],
\end{equation}
for every ${\bf{E}} \in {\rm{Pic}}(X_{P})$, i.e., the curvature of the HYM connection $\nabla$ on ${\bf{E}}$ coincides with the $G$-invariant representative of $c_{1}({\bf{E}})$. 
\end{remark}

As a consequence of Proposition \ref{eigenvalueatorigin}, in \cite{correa2023deformed}, the author proved the following result.

\begin{theorem}
\label{ProofTheo1}
Given a  K\"{a}hler class $[\omega] \in \mathcal{K}(X_{P})$, then for every $[\psi] \in H^{1,1}(X_{P},\mathbbm{R})$ we have 
\begin{equation}
\label{liftphase}
\hat{\Theta}: = {\rm{Arg}} \int_{X_{P}}\frac{(\omega + \sqrt{-1}\psi)^{n}}{n!} = \sum_{\beta \in \Phi_{I}^{+}} \arctan \bigg( \frac{\langle \lambda([\psi]),\beta^{\vee} \rangle}{\langle \lambda([\omega]),\beta^{\vee} \rangle}\bigg) \ (\text{mod} \ 2\pi),
\end{equation}
such that $\lambda([\psi]), \lambda([\omega_{0}]) \in \Lambda_{P} \otimes \mathbbm{R}$. In particular, fixed the unique $G$-invariant representative $\omega_{0} \in [\omega]$, there exists $\phi \in C^{\infty}(X_{P})$, such that $\chi_{\phi}:= \psi + \sqrt{-1}\partial \overline \partial \phi$ satisfies the deformed Hermitian Yang-Mills equation
\begin{equation}
\label{DHYMeqTeo}
{\rm{Im}}\big ( \omega_{0} + \sqrt{-1}\chi_{\phi}\big)^{n} = \tan(\hat{\Theta}) {\rm{Re}}\big ( \omega_{0} + \sqrt{-1}\chi_{\phi}\big)^{n},
\end{equation}
\end{theorem}

\begin{remark} 
\label{dHYMisHYM}
In the above setting, given $[\psi] \in H^{1,1}(X_{P},\mathbbm{R})$, by considering the $G$-invariant representative $\chi \in [\psi]$, it follows that the Lagrangian phase of $\chi$ with respect to some $G$-invariant K\"{a}hler metric $\omega_{0}$ is given by\footnote{In this paper we consider the principal value branch of $\arctan(x)$ given by $(-\frac{\pi}{2},\frac{\pi}{2})$.}
\begin{equation}
\label{lagrangianphase}
\Theta_{\omega_{0}}(\chi) = \sum_{\beta \in \Phi_{I}^{+}} \arctan \bigg( \frac{\langle \lambda([\chi]),\beta^{\vee} \rangle}{\langle \lambda([\omega_{0}]),\beta^{\vee} \rangle}\bigg),
\end{equation}
i.e., we have $\hat{\Theta} = \Theta_{\omega_{0}}(\chi)\ (\text{mod} \ 2\pi)$. In summary, the solution of the dHYM equation in $[\psi] \in H^{1,1}(X_{P},\mathbbm{R})$ is given by the unique $G$-invariant representative of the cohomology class $[\psi]$. In particular, if $[\psi] = c_{1}({\bf{E}})$, for some ${\bf{E}} \in {\rm{Pic}}(X_{P})$, combining this last fact with the ideas described in Remark \ref{primitivecalc}, we have a Hermitian structure ${\bf{h}}$ on ${\bf{E}}$, such that the curvature $F_{\nabla}$ of the associated Chern connection $\nabla$ satisfies:
\begin{enumerate}
\item $\sqrt{-1}\Lambda_{\omega_{0}}(F_{\nabla}) = c \mathbbm{1}_{{\bf{E}}}$;
\item ${\rm{Im}}\big ({\rm{e}}^{-\sqrt{-1}\hat{\Theta}}(\omega_{0} - \frac{1}{2\pi}F_{\nabla})^{n}\big ) = 0$.
\end{enumerate}
In conclusion, $\nabla$ is a HYM connection and a dHYM instanton.
\end{remark}

%\begin{remark}
%From the last proposition we have the following description for the nef cone of a complex flag manifold $X_{P} = G^{\mathbbm{C}}/P$
%\begin{equation}
%{\rm{Nef}}(X_{P}) = \displaystyle \bigoplus_{\alpha \in \Delta \backslash I}\mathbbm{R}_{\geq 0}[\varpi_{\alpha} ].
%\end{equation}
%\end{remark}
\subsection{The first Chern class of flag varieties} In this subsection, we shall review some basic facts related with the Ricci form of $G$-invariant K\"{a}hler metrics on flag varieties. Let $X_{P}$ be a complex flag variety associated with some parabolic Lie subgroup $P = P_{I} \subset G^{\mathbbm{C}}$. By considering the identification $T_{x_{0}}^{1,0}X_{P} \cong \mathfrak{m} \subset \mathfrak{g}^{\mathbbm{C}}$, such that 

\begin{center}
$\mathfrak{m} = \displaystyle \sum_{\alpha \in \Phi_{I}^{-}} \mathfrak{g}_{\alpha}$,
\end{center}
 we can realize $T^{1,0}X_{P}$ as being a holomoprphic vector bundle, associated with the holomorphic $P$-principal bundle $P \hookrightarrow G^{\mathbbm{C}} \to X_{P}$, given by

\begin{center}

$T^{1,0}X_{P} = \Big \{(U_{i})_{i \in J}, \underline{{\rm{Ad}}}\circ \psi_{i j} \colon U_{i} \cap U_{j} \to {\rm{GL}}(\mathfrak{m}) \Big \}$,

\end{center}
where $\underline{{\rm{Ad}}} \colon P \to {\rm{GL}}(\mathfrak{m})$ is the isotropy representation. From this, we obtain 
\begin{equation}
\label{canonicalbundleflag}
{\bf{K}}_{X_{P}}^{-1} = \det \big(T^{1,0}X_{P} \big) = \Big \{(U_{i})_{i \in J}, \det (\underline{{\rm{Ad}}}\circ \psi_{i j}) \colon U_{i} \cap U_{j} \to \mathbbm{C}^{\times} \Big \}.
\end{equation}
Since $P= [P,P] T(\Delta \backslash I)^{\mathbbm{C}}$, considering $\det \circ \underline{{\rm{Ad}}} \in {\text{Hom}}(T(\Delta \backslash I)^{\mathbbm{C}},\mathbbm{C}^{\times})$, we have 
\begin{equation}
\det \underline{{\rm{Ad}}}(\exp({\bf{t}})) = {\rm{e}}^{{\rm{tr}}({\rm{ad}}({\bf{t}})|_{\mathfrak{m}})} = {\rm{e}}^{- \langle \delta_{P},{\bf{t}}\rangle },
\end{equation}
$\forall {\bf{t}} \in {\rm{Lie}}(T(\Delta \backslash I)^{\mathbbm{C}})$, such that $\delta_{P} = \sum_{\alpha \in \Phi_{I}^{+} } \alpha$. Denoting $\vartheta_{\delta_{P}}^{-1} = \det \circ \underline{{\rm{Ad}}}$, it follows that 
\begin{equation}
\label{charactercanonical}
\vartheta_{\delta_{P}} = \displaystyle \prod_{\alpha \in \Delta \backslash I} \vartheta_{\varpi_{\alpha}}^{\langle \delta_{P},\alpha^{\vee} \rangle} \Longrightarrow {\bf{K}}_{X_{P}}^{-1} = \bigotimes_{\alpha \in \Delta \backslash I}\mathscr{O}_{\alpha}(\ell_{\alpha}),
\end{equation}
such that $\ell_{\alpha} = \langle \delta_{P}, \alpha^{\vee} \rangle, \forall \alpha \in \Delta \backslash I$. Notice that $\lambda({\bf{K}}_{X_{P}}^{-1}) = \delta_{P}$, see Eq. (\ref{weightholomorphicvec}). If we consider the invariant K\"{a}hler metric $\rho_{0} \in \Omega^{1,1}(X_{P})^{G}$ defined by
\begin{equation}
\label{riccinorm}
\rho_{0} = \sum_{\alpha \in \Delta \backslash I}2 \pi \langle \delta_{P}, \alpha^{\vee} \rangle {\bf{\Omega}}_{\alpha},
\end{equation}
it follows that
\begin{equation}
\label{ChernFlag}
c_{1}(X_{P}) = \Big [ \frac{\rho_{0}}{2\pi}\Big].
\end{equation}
By the uniqueness of $G$-invariant representative of $c_{1}(X_{P})$, we have 
\begin{center}
${\rm{Ric}}(\rho_{0}) = \rho_{0}$, 
\end{center}
i.e., $\rho_{0} \in \Omega^{1,1}(X_{P})^{G}$ defines a $G$-ivariant K\"{a}hler-Einstein metric on $X_{P}$ (cf. \cite{MATSUSHIMA}). 
\begin{remark}
Given any $G$-invariant K\"{a}hler metric $\omega$ on $X_{P}$, we have ${\rm{Ric}}(\omega) = \rho_{0}$. Thus, it follows that the smooth function $\frac{\det(\omega)}{\det(\rho_{0})}$ is constant. From this, we obtain
\begin{equation}
{\rm{Vol}}(X_{P},\omega) = \frac{1}{n!}\int_{X_{P}}\omega^{n}  =  \frac{\det(\rho_{0}^{-1} \circ \omega)}{n!} \int_{X_{P}}\rho_{0}^{n}.
\end{equation}
Since $\det(\rho_{0}^{-1} \circ \omega) = \frac{1}{(2\pi)^{n}} \prod_{\beta \in \Phi_{I}^{+}}\frac{\langle \lambda([\omega]),\beta^{\vee} \rangle}{\langle \delta_{P},\beta^{\vee} \rangle}$ and $\frac{1}{n!}\int_{X_{P}}c_{1}(X_{P})^{n} = \prod_{\beta \in \Phi_{I}^{+}} \frac{\langle \delta_{P},\beta^{\vee} \rangle}{\langle \varrho^{+},\beta^{\vee}\rangle}$, we conclude that\footnote{cf. \cite{AZAD}.} 
\begin{equation}
\label{VolKahler}
{\rm{Vol}}(X_{P},\omega) = \prod_{\beta \in \Phi_{I}^{+}} \frac{\langle \lambda([\omega]),\beta^{\vee} \rangle}{\langle \varrho^{+},\beta^{\vee} \rangle},
\end{equation}
where $\varrho^{+} = \frac{1}{2} \sum_{\alpha \in \Phi^{+}}\alpha$. Combining the above formula with the ideas introduced in Remark \ref{primitivecalc} we obtain the following expression for the degree of a holomorphic vector bundle ${\bf{E}} \to X_{P}$ with respect to some $G$-invariant K\"{a}hler metric $\omega$ on $X_{P}$:
\begin{equation}
\label{degreeVB}
\deg_{\omega}({\bf{E}}) = \int_{X_{P}}c_{1}({\bf{E}}) \wedge [\omega]^{n-1} = (n-1)!\Bigg [\sum_{\beta \in \Phi_{I}^{+} } \frac{\langle \lambda({\bf{E}}), \beta^{\vee} \rangle}{\langle \lambda([\omega]), \beta^{\vee}\rangle}\Bigg ] \Bigg [\prod_{\beta \in \Phi_{I}^{+}} \frac{\langle \lambda([\omega]),\beta^{\vee} \rangle}{\langle \varrho^{+},\beta^{\vee} \rangle} \Bigg ],
\end{equation}
such that $\lambda({\bf{E}}) \in \Lambda_{P}$, and $\lambda([\omega]) \in \Lambda_{P} \otimes \mathbbm{R}$. 
\end{remark}

\section{Proof of Theorem A}
\label{SecProoftheoA}
In this section, we will prove the following theorem.
\begin{theorem}[Theorem \ref{TheoA}]
\label{thmArestate}
Fixed the unique ${\rm{SU}}(3)$-invariant K\"{a}hler metric $\omega_{0} \in  c_{1}({\mathbbm{P}}(T_{{\mathbbm{P}^{2}}}))$, then there exist holomorphic vector bundles ${\bf{E}}_{j} \to {\mathbbm{P}}(T_{{\mathbbm{P}^{2}}})$, $j=1,2,3$, satisfying the following:
\begin{enumerate}
\item $\rank({\bf{E}}_{j}) > 1$, $j = 1,2,3$;
\item ${\bf{E}}_{1}$ admits a Hermitian structure ${\bf{h}}$ with associated Chern connection $\nabla$ satisfying
\begin{equation}
\begin{cases}\sqrt{-1}\Lambda_{\omega_{0}}(F_{\nabla}) = c\mathbbm{1}_{{\bf{E}}_{1}}, \\ {\rm{Im}}\Big ({\rm{e}}^{-\sqrt{-1}\hat{\Theta}({\bf{E}}_{1})}\Big (\omega_{0} \otimes \mathbbm{1}_{{\bf{E}}_{1}} - \frac{1}{2\pi}F_{\nabla}\Big )^{3}\Big ) =  0,\end{cases}
\end{equation}
such that $c = \frac{\pi}{8} \mu_{[\omega_{0}]}({\bf{E}}_{1})$ and $\hat{\Theta}({\bf{E}}_{1}) = {\rm{Arg}}\int_{{\mathbbm{P}}(T_{{\mathbbm{P}^{2}}})}{\rm{tr}}\Big (\omega_{0} \otimes \mathbbm{1}_{{\bf{E}}_{1}} - \frac{1}{2\pi}F_{\nabla}\Big)^{3} \ (\text{mod} \ 2\pi)$.
\item ${\bf{E}}_{2}$ admits a Hermitian structure ${\bf{h}}$ with associated Chern connection $\nabla$ satisfying 
\begin{equation}
\begin{cases}\sqrt{-1}\Lambda_{\omega_{0}}(F_{\nabla}) = c\mathbbm{1}_{{\bf{E}}_{2}}, \\ {\rm{Im}}\Big ({\rm{e}}^{-\sqrt{-1}\hat{\Theta}({\bf{E}}_{2})}\Big (\omega_{0} \otimes \mathbbm{1}_{{\bf{E}}_{2}} - \frac{1}{2\pi}F_{\nabla}\Big )^{3}\Big ) \neq  0,\end{cases}
\end{equation}
for $c = \frac{\pi}{8} \mu_{[\omega_{0}]}({\bf{E}}_{2})$ and for all $\hat{\Theta}({\bf{E}}_{2}) \in \mathbbm{R}$.
\item ${\bf{E}}_{3}$ admits a Hermitian structure ${\bf{h}}$ with associated Chern connection $\nabla$ satisfying
\begin{equation}
\begin{cases} \sqrt{-1}\Lambda_{\omega_{0}}(F_{\nabla}) \neq c\mathbbm{1}_{{\bf{E}}_{3}}, \\ {\rm{Im}}\Big ({\rm{e}}^{-\sqrt{-1}\hat{\Theta}({\bf{E}}_{3})}\Big (\omega_{0} \otimes \mathbbm{1}_{{\bf{E}}_{3}} - \frac{1}{2\pi}F_{\nabla}\Big )^{3}\Big ) = 0,\end{cases}
\end{equation}
for all $c \in \mathbbm{R}$, such that $\hat{\Theta}({\bf{E}}_{3}) = = {\rm{Arg}}\int_{{\mathbbm{P}}(T_{{\mathbbm{P}^{2}}})}{\rm{tr}}\Big (\omega_{0} \otimes \mathbbm{1}_{{\bf{E}}_{3}} - \frac{1}{2\pi}F_{\nabla}\Big)^{3} \ (\text{mod} \ 2\pi)$.
\end{enumerate}
In particular, ${\bf{E}}_{1}$ and ${\bf{E}}_{2}$ are slope-polystable, and ${\bf{E}}_{3}$ is slope-unstable.
\end{theorem}

The proof which we will present for the above result is constructive, which means that we will construct explicitly the examples of Hermitian connections $\nabla$ on higher rank holomorphic vector bundles ${\bf{E}} \to {\mathbbm{P}}(T_{{\mathbbm{P}^{2}}})$ which illustrate the following cases:
\begin{itemize}
\item {\bf{Type I:}} $\nabla$ is both a HYM connection and a dHYM connection;
\item {\bf{Type II:}} $\nabla$ is a HYM connection but not a dHYM connection;
\item {\bf{Type III:}} $\nabla$ is a dHYM connection but not a HYM connection.
\end{itemize}
As we have seen (Remark \ref{dHYMisHYM}), every line bundle ${\bf{E}} \to {\mathbbm{P}}(T_{{\mathbbm{P}^{2}}})$ admits a Hermitian connection which is Type I. However, explicit examples of Hermitian connections on higher rank holomorphic vector bundles of Type I, or which illustrate Type II and Type III cases, are not known in the literature. In fact, as far as the author knows, no concrete examples of dHYM connection on higher rank holomorphic vector bundles are known so far. It is worth to point out that the first example which illustrate the existence of higher rank dHYM connection was provided in \cite[Example 2.18]{dervan2020z}. 

\subsection{Line bundle case} Consider the complex simple Lie group $G^{\mathbbm{C}} = {\rm{SL}}_{3}(\mathbbm{C})$. In this case, the structure of the associated Lie algebra $\mathfrak{sl}_{3}(\mathbbm{C})$ can be completely determined by means of its Dynkin diagram
\begin{center}
${\dynkin[labels={\alpha_{1},\alpha_{2}},scale=3]A{oo}} $
\end{center}
More precisely, fixed the Cartan subalgebra $\mathfrak{h} \subset \mathfrak{sl}_{3}(\mathbbm{C})$ of diagonal matrices, we have the associated simple root system given by $\Delta  = \{\alpha_{1},\alpha_{2}\}$, such that 
\begin{center}
$\alpha_{j}({\rm{diag}}(d_{1},d_{2},d_{3})) = d_{j} - d_{j+1}$, $j = 1,2$.
\end{center}
$\forall {\rm{diag}}(d_{1},d_{2},d_{3}) \in \mathfrak{h}$. The set of positive roots in this case is given by 
\begin{center}
$\Phi^{+} = \{\alpha_{1}, \alpha_{2}, \alpha_{3} = \alpha_{1} + \alpha_{2}\}$. 
\end{center}
Considering the Cartan-Killing form\footnote{In this case, we have $\kappa(X,Y) = 6{\rm{tr}}(XY), \forall X,Y \in \mathfrak{sl}_{3}(\mathbbm{C})$, see for instance \cite[Chapter 10, \S 4]{procesi2007lie}.} $\kappa(X,Y) = {\rm{tr}}({\rm{ad}}(X){\rm{ad}}(Y)), \forall X,Y \in \mathfrak{sl}_{3}(\mathbbm{C})$, it follows that $\alpha_{j} = \kappa(\cdot,h_{\alpha_{j}})$, $j =1,2,3$, such that\footnote{Notice that $\langle \alpha_{j},\alpha_{j} \rangle = \alpha_{j}(h_{\alpha_{j}}) = \frac{1}{3}, \forall j = 1,2,3.$} 
\begin{equation}
h_{\alpha_{1}} =\frac{1}{6}(E_{11} - E_{22}), \ \ h_{\alpha_{2}} =\frac{1}{6}(E_{22} - E_{33}), \ \ h_{\alpha_{3}} =\frac{1}{6}(E_{11} - E_{33}),
\end{equation}
here we consider the matrices $E_{ij}$ as being the elements of the standard basis of ${\mathfrak{gl}}_{3}(\mathbbm{C})$. Moreover, we have the following relation between simple roots and fundametal weights:
\begin{center}
$\displaystyle{\begin{pmatrix}
\alpha_{1} \\ 
\alpha_{2}
\end{pmatrix} = \begin{pmatrix} \ \ 2 & -1 \\
-1 & \ \ 2\end{pmatrix} \begin{pmatrix}
\varpi_{\alpha_{1}} \\ 
\varpi_{\alpha_{2}}
\end{pmatrix}, \ \ \ \begin{pmatrix}
\varpi_{\alpha_{1}} \\ 
\varpi_{\alpha_{2}}
\end{pmatrix} = \frac{1}{3}\begin{pmatrix} 2 & 1 \\
1 & 2\end{pmatrix} \begin{pmatrix}
\alpha_{1} \\ 
\alpha_{2}
\end{pmatrix}},$
\end{center}
here we consider the Cartan matrix $C = (C_{ij})$ of $\mathfrak{sl}_{3}(\mathbbm{C})$ given by 
\begin{equation}
\label{Cartanmatrix}
C = \begin{pmatrix}
 \ \ 2 & -1 \\
-1 & \ \ 2 
\end{pmatrix}, \ \ C_{ij} = \frac{2\langle \alpha_{i}, \alpha_{j} \rangle}{\langle \alpha_{j}, \alpha_{j} \rangle},
\end{equation}
for more details on the above subject, see for instance \cite{Humphreys}. 
%\begin{figure}[H]
%\includegraphics[scale = .28]{dHYM_Roots_weights_sl_3.jpg}
%\caption{Simple roots, dominant integral elements, and fundamental weights for $\mathfrak{sl}_{3}(\mathbbm{C})$.}
%\label{diagram}
%\end{figure}
Fixed the standard Borel subgroup $B \subset {\rm{SL}}_{3}(\mathbbm{C})$, i.e.,
\begin{center}
$B = \Bigg \{ \begin{pmatrix} \ast & \ast & \ast \\
0 & \ast & \ast \\
0 & 0 & \ast \end{pmatrix} \in {\rm{SL}}_{3}(\mathbbm{C})\Bigg\},$
\end{center}
we consider the flag variety obtained from $I = \emptyset$, i.e., the homogeneous Fano threefold given by the Wallach space ${\mathbbm{P}}(T_{{\mathbbm{P}^{2}}}) = {\rm{SL}}_{3}(\mathbbm{C})/B$. In this particular case, we have the following:
\begin{enumerate}
\item[(i)] $H^{2}({\mathbbm{P}}(T_{{\mathbbm{P}^{2}}}),\mathbbm{R}) = H^{1,1}({\mathbbm{P}}(T_{{\mathbbm{P}^{2}}}),\mathbbm{R}) = \mathbbm{R}[{\bf{\Omega}}_{\alpha_{1}}] \oplus \mathbbm{R}[{\bf{\Omega}}_{\alpha_{2}}]$;
\item[(ii)] ${\rm{Pic}}({\mathbbm{P}}(T_{{\mathbbm{P}^{2}}})) = \Big \{ \mathscr{O}_{\alpha_{1}}(s_{1}) \otimes \mathscr{O}_{\alpha_{2}}(s_{2}) \ \Big | \ s_{1}, s_{2} \in \mathbbm{Z}\Big \}$.
\end{enumerate}
Let $\omega_{0}$ be the unique ${\rm{SU}}(3)$-invariant K\"{a}hler metric on ${\mathbbm{P}}(T_{{\mathbbm{P}^{2}}})$, such that $[\omega_{0}] = c_{1}({\mathbbm{P}}(T_{{\mathbbm{P}^{2}}}))$\footnote{It is worth pointing out that there is nothing special with this choice. In fact, all the computations presented in this example work for an arbitrary choice of ${\rm{SU}}(3)$-invariant (integral) K\"{a}hler class on ${\mathbbm{P}}(T_{{\mathbbm{P}^{2}}})$.}. Since $\lambda({\bf{K}}_{{\mathbbm{P}}(T_{{\mathbbm{P}^{2}}})}^{-1}) = \delta_{B} = 2(\varpi_{\alpha_{1}} + \varpi_{\alpha_{2}})$, from Eq. (\ref{ChernFlag}), it follows that 
\begin{equation}
\omega_{0} = \langle \delta_{B}, \alpha_{1}^{\vee} \rangle {\bf{\Omega}}_{\alpha_{1}} + \langle \delta_{B}, \alpha_{2}^{\vee} \rangle {\bf{\Omega}}_{\alpha_{2}} = 2 \big ({\bf{\Omega}}_{\alpha_{1}} + {\bf{\Omega}}_{\alpha_{2}}\big),
\end{equation}
in particular, notice that $\lambda([\omega_{0}]) = \delta_{B} = 2\varrho^{+}$, thus
\begin{equation}
{\rm{Vol}}({\mathbbm{P}}(T_{{\mathbbm{P}^{2}}}),\omega_{0}) = \prod_{j = 1}^{3}\frac{\langle \delta_{B},\alpha_{j}^{\vee} \rangle}{\langle \varrho^{+},\alpha_{j}^{\vee} \rangle} = 8,
\end{equation}
see Eq. (\ref{VolKahler}). Given any $[\psi] \in H^{1,1}({\mathbbm{P}}(T_{{\mathbbm{P}^{2}}}),\mathbbm{R})$, from Theorem \ref{ProofTheo1}, we have
\begin{equation}
\hat{\Theta} = {\rm{Arg}}\int_{{\mathbbm{P}}(T_{{\mathbbm{P}^{2}}})}(\omega_{0} + \sqrt{-1}\psi)^{3} = \sum_{j = 1}^{3} \arctan \bigg( \frac{\langle \lambda([\psi]),\alpha_{j}^{\vee} \rangle}{\langle \delta_{B},\alpha_{j}^{\vee} \rangle}\bigg) \ (\text{mod} \ 2\pi),
\end{equation}
notice that, since $I = \emptyset$, it follows that  $\Phi_{I}^{+} = \Phi^{+} = \{\alpha_{1}, \alpha_{2}, \alpha_{3} = \alpha_{1} + \alpha_{2}\}$. Therefore, if we suppose that $[\psi] = s_{1}[{\bf{\Omega}}_{\alpha_{1}}] + s_{2}[{\bf{\Omega}}_{\alpha_{2}}]$, for some $s_{1},s_{2} \in \mathbbm{R}$, by considering the Cartan matrix $C = (C_{ij})$ of $\mathfrak{sl}_{3}(\mathbbm{C})$ (see Eq. (\ref{Cartanmatrix})), we obtain the following:
\begin{enumerate}
\item $\langle \delta_{B},\alpha_{1}^{\vee} \rangle = \langle \delta_{B},\alpha_{2}^{\vee} \rangle = 2$ and $\langle \delta_{B},\alpha_{3}^{\vee} \rangle = 4$;
\item $\langle \lambda([\psi]),\alpha_{1}^{\vee} \rangle = s_{1}$, $\langle \lambda([\psi]),\alpha_{2}^{\vee} \rangle = s_{2}$, $\langle \lambda([\psi]),\alpha_{3}^{\vee} \rangle = s_{1} + s_{2}.$
\end{enumerate}
From above, we conclude that 
\begin{equation}
\label{lagrangianphasesu3}
\hat{\Theta}= \arctan \bigg ( \frac{s_{1}}{2}\bigg) + \arctan \bigg ( \frac{s_{2}}{2}\bigg) + \arctan \bigg(\frac{s_{1} + s_{2}}{4} \bigg) \ (\text{mod} \ 2\pi).
\end{equation}
From Eq. (\ref{lagrangianphase}), given an arbitrary ${\rm{SU}}(3)$-invariant $(1,1)$-form $\chi = s_{1}{\bf{\Omega}}_{\alpha_{1}} + s_{2}{\bf{\Omega}}_{\alpha_{2}}$, we have the following concrete expression for its Lagrangian phase w.r.t. $\omega_{0}$:
\begin{equation}
\label{phaseGclass}
\Theta_{\omega_{0}}(\chi) = \arctan \bigg ( \frac{s_{1}}{2}\bigg) + \arctan \bigg ( \frac{s_{2}}{2}\bigg) + \arctan \bigg(\frac{s_{1} + s_{2}}{4} \bigg).
\end{equation}
Also, as we have seen in Remark \ref{dHYMisHYM}, if $[\chi] = c_{1}({\bf{E}})$, for some ${\bf{E}} \in {\rm{Pic}}({\mathbbm{P}}(T_{{\mathbbm{P}^{2}}}))$, then we have a Hermitian metric ${\bf{h}}$ on ${\bf{E}}$ for which that associated Chern connection $\nabla$ satisfies
\begin{enumerate}
\item $\sqrt{-1}\Lambda_{\omega_{0}}(F_{\nabla}) = c \mathbbm{1}_{{\bf{E}}}$ (HYM);
\item ${\rm{Im}}\big ({\rm{e}}^{-\sqrt{-1}\hat{\Theta}}(\omega_{0} - \frac{1}{2\pi}F_{\nabla})^{3}\big ) = 0$ (dHYM).
\end{enumerate}
\begin{remark}
\label{HYMinstantonlinebundle}
Let us describe $\nabla$ explicitly. From Proposition \ref{C8S8.2Sub8.2.3P8.2.6}, we have 
\begin{equation}
{\bf{E}} = \mathscr{O}_{\alpha_{1}}(s_{1}) \otimes \mathscr{O}_{\alpha_{2}}(s_{2}),
\end{equation}
such that $s_{1},s_{1} \in \mathbbm{Z}$. Given an open set $U \subset {\mathbbm{P}}(T_{{\mathbbm{P}^{2}}})$ which trivializes both ${\bf{E}} \to {\mathbbm{P}}(T_{{\mathbbm{P}^{2}}})$ and $B \hookrightarrow {\rm{SL}}_{3}(\mathbbm{C}) \to {\mathbbm{P}}(T_{{\mathbbm{P}^{2}}})$, and denoting by $w$ the fiber coordinate in ${\bf{E}}|_{U}$, we can construct a Hermitian structure ${\bf{h}}$ on ${\bf{E}}$ by gluing the local Hermitian structures 
\begin{equation}
{\bf{h}}_{U} = \frac{w\overline{w}} {||s_{U}v_{\varpi_{\alpha_{1}}}^{+}||^{2s_{1}}||s_{U}v_{\varpi_{\alpha_{2}}}^{+}||^{2s_{2}}},
\end{equation}
where $s_{U} \colon U \subset {\mathbbm{P}}(T_{{\mathbbm{P}^{2}}})\to {\rm{SL}}_{3}(\mathbbm{C})$ is some local section, here we consider $||\cdot||$ defined by some fixed ${\rm{SU}}(3)$-invariant inner product on $V(\varpi_{\alpha_{k}})$, $k = 1,2$. From this, we can describe the associated Chern connection $\nabla$ (locally) by
\begin{equation}
\nabla|_{U} = {\rm{d}} + A_{U},
\end{equation}
where
\begin{equation}
A_{U} = - \partial \log \Big ( ||s_{U}v_{\varpi_{\alpha_{1}}}^{+}||^{2s_{1}}||s_{U}v_{\varpi_{\alpha_{2}}}^{+}||^{2s_{2}}\Big).
\end{equation}
In particular, consider $U = U^{-}(B) \subset {\mathbbm{P}}(T_{{\mathbbm{P}^{2}}})$, such that
\begin{equation}
U^{-}(B) = \Bigg \{ \begin{pmatrix}
1 & 0 & 0 \\
z_{1} & 1 & 0 \\                  
z_{2}  & z_{3} & 1
 \end{pmatrix}B \ \Bigg | \ z_{1},z_{2},z_{3} \in \mathbbm{C} \Bigg \} \ \ \ ({\text{opposite big cell}}),
\end{equation}
The open set above is dense and contractible, so it trivializes the desired bundles over ${\mathbbm{P}}(T_{{\mathbbm{P}^{2}}})$. By taking the local section $s_{U} \colon U^{-}(B) \to {\rm{SL}}_{3}(\mathbbm{C})$, such that $s_{U}(nB) = n$, $\forall nB \in U^{-}(B)$, and considering
\begin{center}
$V(\varpi_{\alpha_{1}}) = \mathbbm{C}^{3}$ \ \  and \ \ $V(\varpi_{\alpha_{2}}) = \bigwedge^{2}(\mathbbm{C}^{3}),$
\end{center}
where $v_{\varpi_{\alpha_{1}}}^{+} = e_{1}$, and $v_{\varpi_{\alpha_{2}}}^{+} = e_{1} \wedge e_{2}$, fixed $||\cdot||$ defined by the standard ${\rm{SU}}(3)$-invariant inner product on $\mathbbm{C}^{3}$ and $\bigwedge^{2}(\mathbbm{C}^{3})$, we obtain
\begin{equation}
A_{U^{-}(B)} = - \partial \log \Bigg [ \bigg ( 1 + \sum_{i = 1}^{2}|z_{i}|^{2} \bigg )^{s_{1}} \bigg (1 + |z_{3}|^{2} + \bigg | \det \begin{pmatrix}
z_{1} & 1  \\                  
z_{2}  & z_{3} 
 \end{pmatrix} \bigg |^{2} \bigg )^{s_{2}} \Bigg ],
\end{equation}
From above we obtain an explicit example of HYM connection which is also a dHYM connection on ${\bf{E}} = \mathscr{O}_{\alpha_{1}}(s_{1}) \otimes \mathscr{O}_{\alpha_{2}}(s_{2})$. 
\end{remark}

\subsection{dHYM connections and polystable holomorphc vector bundles} 
\label{dHYM_polystable}
Since every line bundle ${\bf{L}} \to {\mathbbm{P}}(T_{{\mathbbm{P}^{2}}})$ admits a dHYM connection, a trivial way to construct higher rank dHYM connections is taking Whitney sums of the form:
\begin{equation}
{\bf{E}} := \underbrace{{\bf{L}} \oplus \cdots \oplus {\bf{L}}}_{r-{\text{times}}},
\end{equation}
In the above setting, the dHYM connection on ${\bf{L}}$ induces a Hermitian connection $\nabla$ on ${\bf{E}}$, such that 
\begin{enumerate}
\item $\nabla^{0,1} = \overline{\partial}$;
\item $\sqrt{-1}\Lambda_{\omega_{0}}(F_{\nabla}) = c \mathbbm{1}_{{\bf{E}}}$ (HYM);
\item ${\rm{Im}}\big ({\rm{e}}^{-\sqrt{-1}\hat{\Theta}({\bf{E}})}(\omega_{0} \otimes \mathbbm{1}_{{\bf{E}}} - \frac{1}{2\pi}F_{\nabla})^{3}\big ) = 0$ (dHYM).
\end{enumerate}
The above fact is a consequence of the ideas presented in Remark \ref{dHYMisHYM}. Notice that, in the above setting, we have 
\begin{equation}
\hat{\Theta}({\bf{E}}) = {\rm{Arg}}\int_{{\mathbbm{P}}(T_{{\mathbbm{P}^{2}}})}{\rm{tr}}\Big (\omega_{0} \otimes \mathbbm{1}_{{\bf{E}}} - \frac{1}{2\pi}F_{\nabla}\Big)^{3} = \hat{\Theta}({\bf{L}}) \ (\text{mod} \ 2\pi),
\end{equation}
where  
\begin{equation}
\label{phaselinebundle}
\hat{\Theta}({\bf{L}}) = \underbrace{\arctan \bigg ( \frac{s_{1}}{2}\bigg) + \arctan \bigg ( \frac{s_{2}}{2}\bigg) + \arctan \bigg(\frac{s_{1} + s_{2}}{4} \bigg)}_{\Theta_{\omega_{0}}(\chi_{\bf{L}})} \ (\text{mod} \ 2\pi),
\end{equation}
such that $s_{j} = \langle \lambda({\bf{L}}),\alpha_{j}^{\vee} \rangle$, $j = 1,2$, and $[\chi_{{\bf{L}}}] = c_{1}({\bf{L}})$, see Eq. (\ref{weightholomorphicvec}) and Eq. (\ref{phaseGclass}). 
\begin{remark}
\label{trivialcase}
From above, we have several examples of higher rank dHYM instantons. We shall refer to this class of examples as trivial dHYM instantons.
\end{remark}
In order to construct examples of solutions to dHYM equation which are non-trivial we proceed in the following way. Given some real number $m \in \mathbbm{R}$, we define the following subsets of the Picard group ${\rm{Pic}}({\mathbbm{P}}(T_{{\mathbbm{P}^{2}}}))$:
\begin{itemize}
\item $\mathcal{D}_{m}(\omega_{0}):= \big \{ {\bf{L}} \in {\rm{Pic}}({\mathbbm{P}}(T_{{\mathbbm{P}^{2}}})) \ \big | \ \Lambda_{\omega_{0}}(\chi_{{\bf{L}}}) = m\big \}$,
\item $\mathcal{L}_{m}(\omega_{0}):= \big \{ {\bf{L}} \in {\rm{Pic}}({\mathbbm{P}}(T_{{\mathbbm{P}^{2}}})) \ \big | \ \Theta_{\omega_{0}}(\chi_{\bf{L}}) = m\big \}$,
\end{itemize}
such that $\chi_{{\bf{L}}} \in c_{1}({\bf{L}})$ denotes the associated ${\rm{SU}}(3)$-invariant representative. Notice that $\mathcal{D}_{m}(\omega_{0})$ and $\mathcal{L}_{m}(\omega_{0})$ can be described, respectively, by the following concrete equations:
\begin{itemize}
\item Given ${\bf{L}} = \mathscr{O}_{\alpha_{1}}(s_{1}) \otimes  \mathscr{O}_{\alpha_{2}}(s_{2}) \in \mathcal{D}_{m}(\omega_{0})$, then 
\begin{equation}
\label{eqHYM}
\Lambda_{\omega_{0}}(\chi_{{\bf{L}}}) = m \iff \frac{3}{4}(s_{1} + s_{2}) = m;
\end{equation}
\item Given ${\bf{L}} = \mathscr{O}_{\alpha_{1}}(s_{1}) \otimes  \mathscr{O}_{\alpha_{2}}(s_{2}) \in \mathcal{L}_{m}(\omega_{0})$, then 
\begin{equation}
\Theta_{\omega_{0}}(\chi_{\bf{L}}) = m \iff \arctan \bigg ( \frac{s_{1}}{2}\bigg) + \arctan \bigg ( \frac{s_{2}}{2}\bigg) + \arctan \bigg(\frac{s_{1} + s_{2}}{4} \bigg) = m.
\end{equation}
\end{itemize}
From above, in particular, one can check that $m = 0 \Rightarrow \mathcal{D}_{0}(\omega_{0}) = \mathcal{L}_{0}(\omega_{0})$. Moreover, since
\begin{equation}
\deg_{\omega_{0}}({\bf{L}}) = \int_{{\mathbbm{P}}(T_{{\mathbbm{P}^{2}}})} \chi_{{\bf{L}}} \wedge \omega_{0}^{2} = \frac{1}{3}\Lambda_{\omega_{0}}(\chi_{{\bf{L}}})\int_{{\mathbbm{P}}(T_{{\mathbbm{P}^{2}}})}\omega_{0}^{3} = \frac{3!}{3}\Lambda_{\omega_{0}}(\chi_{{\bf{L}}}) {\rm{Vol}}({\mathbbm{P}}(T_{{\mathbbm{P}^{2}}}),\omega_{0}),
\end{equation}
see Eq. (\ref{degreeVB}), we conclude that 
\begin{equation}
\mathcal{D}_{0}(\omega_{0}) = \mathcal{L}_{0}(\omega_{0}) = \underbrace{\Big \{ {\bf{L}} \in {\rm{Pic}}({\mathbbm{P}}(T_{{\mathbbm{P}^{2}}})) \ \big | \  \deg_{\omega_{0}}({\bf{L}}) = 0\Big \}}_{ = {\rm{Pic}}_{\omega_{0}}^{0}({\mathbbm{P}}(T_{{\mathbbm{P}^{2}}}))}.
\end{equation}
Therefore, it follows that $\mathcal{D}_{0}(\omega_{0}) = \mathcal{L}_{0}(\omega_{0})$ is the subgroup 
\begin{center}
${\rm{Pic}}_{\omega_{0}}^{0}({\mathbbm{P}}(T_{{\mathbbm{P}^{2}}})) \subset {\rm{Pic}}({\mathbbm{P}}(T_{{\mathbbm{P}^{2}}}))$.
\end{center}
From Eq. (\ref{eqHYM}), we have the following description in terms of generators
\begin{equation}
\mathcal{D}_{0}(\omega_{0}) = \mathcal{L}_{0}(\omega_{0}) = {\rm{Pic}}_{\omega_{0}}^{0}({\mathbbm{P}}(T_{{\mathbbm{P}^{2}}})) = \Big \langle  \mathscr{O}_{\alpha_{1}}(1) \otimes  \mathscr{O}_{\alpha_{2}}(-1)\Big \rangle.
\end{equation}
Hence, taking distinct elements ${\bf{L}}_{1},\ldots,{\bf{L}}_{r} \in {\rm{Pic}}_{\omega_{0}}^{0}({\mathbbm{P}}(T_{{\mathbbm{P}^{2}}}))$, one can define
\begin{equation}
{\bf{E}} := {\bf{L}}_{1} \oplus \cdots \oplus {\bf{L}}_{r}.
\end{equation}
Denoting by $\chi_{j}$, the {\rm{SU}}(3)-invariant representative of $c_{1}({\bf{L}}_{j})$, $\forall j =1,\ldots,r$, we have an induced Hermitian structure ${\bf{h}}$ on ${\bf{E}}$, such that the curvature of the associated Chern connection $\nabla$ satisfies
\begin{equation}
\frac{\sqrt{-1}}{2\pi}F_{\nabla} = \begin{pmatrix} \chi_{1} & \cdots & 0 \\
 \vdots & \ddots & \vdots \\
 0 & \cdots & \chi_{r}\end{pmatrix}.
\end{equation}
By construction, we have 
\begin{equation}
\sqrt{-1}\Lambda_{\omega_{0}}(F_{\nabla}) = 0,
\end{equation}
i.e., $\nabla$ is an HYM connection, see Remark \ref{primitivecalc}. In particular, notice that $\mu_{[\omega_{0}]}({\bf{E}}) = 0$. Moreover, since
\begin{equation}
\Big (\omega_{0} \otimes \mathbbm{1}_{{\bf{E}}} - \frac{1}{2\pi}F_{\nabla} \Big )^{3} = \begin{pmatrix} (\omega_{0} + \sqrt{-1}\chi_{1})^{3} & \cdots & 0 \\
 \vdots & \ddots & \vdots \\
 0 & \cdots & (\omega_{0} + \sqrt{-1}\chi_{r})^{3}\end{pmatrix},
\end{equation}
and $\Theta_{\omega_{0}}(\chi_{j}) = 0$, $\forall j = 1,\ldots,r$, it follows from Theorem \ref{ProofTheo1} (see also Remark \ref{dHYMisHYM}) that 
\begin{equation}
{\rm{Im}}\Big ( {\rm{e}}^{-\sqrt{-1} \hat{\Theta}({\bf{E}})}\big (\omega_{0} \otimes \mathbbm{1}_{{\bf{E}}} - \frac{1}{2\pi}F_{\nabla} \big )^{3} \Big ) = 0,
\end{equation}
such that 
\begin{equation}
\hat{\Theta}({\bf{E}}) = {\rm{Arg}}\int_{{\mathbbm{P}}(T_{{\mathbbm{P}^{2}}})}{\rm{tr}}\Big (\omega_{0} \otimes \mathbbm{1}_{{\bf{E}}} - \frac{1}{2\pi}F_{\nabla}\Big)^{3} = 0 \ (\text{mod} \ 2\pi)
\end{equation}
Thus, we have that $\nabla$ is a non-trivial example of dHYM connection. The class of examples presented above illustrates the higher rank examples of Hermitian connections of Type I. The above construction can be summarized in the following lemma.
\begin{lemma}
\label{L1}
There exists a Hermitian homlomorphic vector bundle $({\bf{E}},{\bf{h}}) \to ({\mathbbm{P}}(T_{{\mathbbm{P}^{2}}}),\omega_{0})$, with $\rank({\bf{E}}) > 1$, such that the associated Chern connection $\nabla$ satisfies 
\begin{equation}
\begin{cases}\sqrt{-1}\Lambda_{\omega_{0}}(F_{\nabla}) = c\mathbbm{1}_{{\bf{E}}}, \\ {\rm{Im}}\Big ({\rm{e}}^{-\sqrt{-1}\hat{\Theta}({\bf{E}})}\Big (\omega_{0} \otimes \mathbbm{1}_{{\bf{E}}} - \frac{1}{2\pi}F_{\nabla}\Big )^{3}\Big ) =  0,\end{cases}
\end{equation}
where $c = \frac{\pi}{8} \mu_{[\omega_{0}]}({\bf{E}})$ and $\hat{\Theta}({\bf{E}}) = {\rm{Arg}}\int_{{\mathbbm{P}}(T_{{\mathbbm{P}^{2}}})}{\rm{tr}}\Big (\omega_{0} \otimes \mathbbm{1}_{{\bf{E}}} - \frac{1}{2\pi}F_{\nabla}\Big)^{3} \ (\text{mod} \ 2\pi)$. In particular, we have that ${\bf{E}}$ is slope-polystable.
\end{lemma}
\begin{remark}
Proceeding as in Remark \ref{HYMinstantonlinebundle}, we can describe $\nabla$ obtained above in an explicit way. In fact, considering the opposite big cell $U^{-}(B) \subset {\mathbbm{P}}(T_{{\mathbbm{P}^{2}}})$, and denoting 
\begin{equation}
{\bf{L}}_{k} = \mathscr{O}_{\alpha_{1}}(\ell_{k}) \otimes  \mathscr{O}_{\alpha_{2}}(-\ell_{k}), 
\end{equation}
such that $\ell_{k} \in \mathbbm{Z}$, $k = 1,\ldots,r$, we have the dHYM connection given (locally) by 
\begin{equation}
\nabla|_{U^{-}(B)} = {\rm{d}} + \begin{pmatrix} A_{U^{-}(B)}^{(1)} & \cdots & 0 \\
 \vdots & \ddots & \vdots \\
 0 & \cdots & A_{U^{-}(B)}^{(r)}\end{pmatrix},
\end{equation}
where 
\begin{equation}
A_{U^{-}(B)}^{(k)} = - \partial \log \Bigg [ \frac{\bigg ( 1 + \displaystyle{\sum_{i = 1}^{2}|z_{i}|^{2}} \bigg )^{\ell_{k}}}{ \bigg (1 + |z_{3}|^{2} + \bigg | \det \begin{pmatrix}
z_{1} & 1  \\                  
z_{2}  & z_{3} 
 \end{pmatrix} \bigg |^{2} \bigg )^{\ell_{k}}} \Bigg ],
\end{equation}
for all $k = 1,\ldots,r$.
\end{remark}

Let us now construct examples of Hermitian connections of Type II on holomorphic vector bundles of rank 2. Let $m = \frac{3}{4}$, and consider $\mathcal{D}_{\frac{3}{4}}(\omega_{0})$. By construction, we can take ${\bf{F}}, {\bf{G}} \in \mathcal{D}_{\frac{3}{4}}(\omega_{0})$, such that 
\begin{equation}
\label{exampleHYMnotdHYM}
{\bf{F}} = \mathscr{O}_{\alpha_{1}}(2) \otimes  \mathscr{O}_{\alpha_{2}}(-1) \ \ \ {\text{and}} \ \ \ {\bf{G}} = \mathscr{O}_{\alpha_{1}}(3) \otimes  \mathscr{O}_{\alpha_{2}}(-2).
\end{equation}
From above, it follows that 
\begin{equation}
\Lambda_{\omega_{0}}(\chi_{{\bf{F}}}) = \Lambda_{\omega_{0}}(\chi_{{\bf{G}}}) = \frac{3}{4} \ \ \ \ {\text{and}} \ \ \ \ \frac{\Theta_{\omega_{0}}(\chi_{\bf{F}}) - \Theta_{\omega_{0}}(\chi_{\bf{G}})}{2\pi} \notin \mathbbm{Z}. 
\end{equation}
In particular, we have 
\begin{equation}
\mu_{[\omega_{0}]}({\bf{F}}) = \mu_{[\omega_{0}]}({\bf{G}}) = 12.
\end{equation}
If we define ${\bf{E}} = {\bf{F}} \oplus {\bf{G}}$, it follows that there exits a Hermitian structure ${\bf{h}}$ on ${\bf{E}}$, such that the curvature of the associated Chern connection $\nabla$ is given by
\begin{equation}
\frac{\sqrt{-1}}{2\pi}F_{\nabla} = \begin{pmatrix} \chi_{{\bf{F}}} & 0 \\
 0 & \chi_{{\bf{G}}} \end{pmatrix}.
\end{equation}
From the ideas introduced in Remark \ref{HYMinstantonlinebundle}, the Hermitian connection $\nabla$ on ${\bf{E}} = {\bf{F}} \oplus {\bf{G}}$ mentioned above can described (locally) by  
\begin{equation}
\nabla|_{U^{-}(B)} = {\rm{d}} + \begin{pmatrix} A_{{\bf{F}}} & 0 \\
 0 & A_{{\bf{G}}} \end{pmatrix},
\end{equation}
such that 
\begin{itemize}
\item $\displaystyle{A_{{\bf{F}}} = - \partial \log \Bigg [ \frac{\bigg ( 1 + \displaystyle{\sum_{i = 1}^{2}|z_{i}|^{2}} \bigg )^{2}}{ \bigg (1 + |z_{3}|^{2} + \bigg | \det \begin{pmatrix}
z_{1} & 1  \\                  
z_{2}  & z_{3} 
 \end{pmatrix} \bigg |^{2} \bigg )} \Bigg ]},$
 \item $\displaystyle{A_{{\bf{G}}}=  - \partial \log \Bigg [ \frac{\bigg ( 1 + \displaystyle{\sum_{i = 1}^{2}|z_{i}|^{2}} \bigg )^{3}}{ \bigg (1 + |z_{3}|^{2} + \bigg | \det \begin{pmatrix}
z_{1} & 1  \\                  
z_{2}  & z_{3} 
 \end{pmatrix} \bigg |^{2} \bigg )^{2}} \Bigg ]}.$
\end{itemize}
By construction, from Theorem \ref{ProofTheo1} (see Remark \ref{dHYMisHYM}), we have 
\begin{enumerate}
\item $\sqrt{-1}\Lambda_{\omega_{0}}(F_{\nabla}) = \frac{3\pi}{2} \mathbbm{1}_{{\bf{E}}} = \mu_{[\omega_{0}]}({\bf{E}})\frac{\pi}{8}\mathbbm{1}_{{\bf{E}}}$, notice that $\mu_{[\omega_{0}]}({\bf{E}}) = 12$;
\item ${\rm{Im}}\big ({\rm{e}}^{-\sqrt{-1}\hat{\Theta}({\bf{E}})}(\omega_{0} \otimes \mathbbm{1}_{{\bf{E}}} - \frac{1}{2\pi}F_{\nabla})^{3}\big ) \neq 0$, $\forall \hat{\Theta}({\bf{E}}) \in \mathbbm{R}$.
\end{enumerate}
Therefore, $\nabla$ is an example of HYM connection which is not a dHYM instanton, i.e., an example of Hermitian connection of Type II. In summary, we have the following.

\begin{lemma}
\label{L2}
There exists a Hermitian homlomorphic vector bundle $({\bf{E}},{\bf{h}}) \to ({\mathbbm{P}}(T_{{\mathbbm{P}^{2}}}),\omega_{0})$, with $\rank({\bf{E}}) > 1$, such that the associated Chern connection $\nabla$ satisfies 
\begin{equation}
\begin{cases}\sqrt{-1}\Lambda_{\omega_{0}}(F_{\nabla}) = c\mathbbm{1}_{{\bf{E}}}, \\ {\rm{Im}}\Big ({\rm{e}}^{-\sqrt{-1}\hat{\Theta}({\bf{E}})}\Big (\omega_{0} \otimes \mathbbm{1}_{{\bf{E}}} - \frac{1}{2\pi}F_{\nabla}\Big )^{3}\Big ) \neq  0,\end{cases}
\end{equation}
for $c = \frac{\pi}{8} \mu_{[\omega_{0}]}({\bf{E}})$ and for all $\hat{\Theta}({\bf{E}}) \in \mathbbm{R}$. In particular, we have that ${\bf{E}}$ is slope-polystable.
\end{lemma}

\begin{remark}
Notice that $\mathcal{D}_{\frac{3}{4}}(\omega_{0})$ is the set of line bundles (up to isomorphism) of the form
\begin{equation}
{\bf{L}} = \mathscr{O}_{\alpha_{1}}(s) \otimes  \mathscr{O}_{\alpha_{2}}(1-s), \ \ s \in \mathbbm{Z}.
\end{equation}
Therefore, given some integer $r > 0$, if we define
\begin{equation}
{\bf{E}}:= \bigoplus_{s = 2}^{r+1} \Big (\mathscr{O}_{\alpha_{1}}(s) \otimes  \mathscr{O}_{\alpha_{2}}(1-s) \Big ),
\end{equation}
by following a similar argument as in the case described in Eq. (\ref{exampleHYMnotdHYM}), one can construct examples of higher rank HYM connections of Type II.
\end{remark}

\subsection{dHYM connections on unstable holomorphic vector bundles} 
\label{Unstableconstruction}
In this subsection, we will construct some examples of dHYM connections which do not satisfy the HYM equation. In order to do so, we proceed in the following way. Let $m_{1},m_{2},m_{3} \in \mathbbm{R}$, such that: 
\begin{enumerate}
\item $m_{2} \neq m_{3}$;
\item $\mathcal{L}_{m_{1}}(\omega_{0}) \neq \emptyset $ and  $\mathcal{D}_{m_{i}}(\omega_{0}) \neq \emptyset $, $i = 2,3$;
\item $\mathcal{L}_{m_{1}}(\omega_{0}) \cap \mathcal{D}_{m_{i}}(\omega_{0}) \neq \emptyset$, $i =2,3$.
\end{enumerate}
Given $m_{1},m_{2},m_{3} \in \mathbbm{R}$ satisfying the above conditions, let ${\bf{F}},{\bf{G}} \in \mathcal{L}_{m_{1}}(\omega_{0})$, such that 
\begin{equation}
{\bf{F}} \in \mathcal{D}_{m_{2}}(\omega_{0}) \ \ {\text{and}} \ \ {\bf{G}} \in \mathcal{D}_{m_{3}}(\omega_{0}). 
\end{equation}
From above, we define 
\begin{equation}
{\bf{E}} = {\bf{F}} \oplus {\bf{G}}.
\end{equation}
By construction, we have that 
\begin{equation}
\deg_{\omega_{0}}({\bf{F}}) = \frac{3!}{3}m_{2} {\rm{Vol}}({\mathbbm{P}}(T_{{\mathbbm{P}^{2}}}),\omega_{0}) \ \ \ \text{and} \ \ \ \deg_{\omega_{0}}({\bf{G}}) = \frac{3!}{3}m_{3} {\rm{Vol}}({\mathbbm{P}}(T_{{\mathbbm{P}^{2}}}),\omega_{0}),
\end{equation}
Thus, it follows that 
\begin{equation}
\mu_{[\omega_{0}]}({\bf{F}}) = \deg_{\omega_{0}}({\bf{F}}) \neq \deg_{\omega_{0}}({\bf{G}}) = \mu_{[\omega_{0}]}({\bf{G}}),
\end{equation}
so ${\bf{E}}$ is not slope-semistable (e.g. \cite{Kobayashi+1987}). Since a slope-polystable holomorphic vector bundle is in particular slope-semistable, from Kobayashi-Hitchin correspondence \cite{donaldson1985anti,donaldson1987infinite},\cite{uhlenbeck1986existence}, it follows that ${\bf{E}}$ does not admit a HYM connection. On the other hand, since ${\bf{F}}, {\bf{G}} \in \mathcal{L}_{m_{1}}(\omega_{0})$, it follows that 
\begin{equation}
\Theta_{\omega_{0}}(\chi_{{\bf{F}}}) = \Theta_{\omega_{0}}(\chi_{{\bf{G}}}) = m_{1}. 
\end{equation}
where $\chi_{{\bf{F}}} \in c_{1}({\bf{F}})$ and $\chi_{{\bf{G}}} \in c_{1}({\bf{G}})$ are the ${\rm{SU}}(3)$-invariant representatives. Hence, it follows that there exists a Hermitian structure ${\bf{h}}$ on ${\bf{E}}$, such that the curvature of the associated Chern connection $\nabla$ is given by
\begin{equation}
\frac{\sqrt{-1}}{2\pi}F_{\nabla} = \begin{pmatrix} \chi_{{\bf{F}}} & 0 \\
 0 & \chi_{{\bf{G}}} \end{pmatrix}.
\end{equation}
Since
\begin{equation}
{\rm{tr}}\Big (\omega_{0} \otimes \mathbbm{1}_{{\bf{E}}} - \frac{1}{2\pi}F_{\nabla}\Big)^{3} = \big (\omega_{0} + \sqrt{-1}\chi_{{\bf{F}}} \big )^{3} + \big (\omega_{0} + \sqrt{-1}\chi_{{\bf{G}}} \big )^{3},
\end{equation}
it follows that 
\begin{equation}
\hat{\Theta}({\bf{E}}) = {\rm{Arg}}\int_{{\mathbbm{P}}(T_{{\mathbbm{P}^{2}}})}{\rm{tr}}\Big (\omega_{0} \otimes \mathbbm{1}_{{\bf{E}}} - \frac{1}{2\pi}F_{\nabla}\Big)^{3} = m_{1} \ (\text{mod} \ 2\pi).
\end{equation}
Hence, from Theorem \ref{ProofTheo1} (see also Remark \ref{dHYMisHYM}), we obtain
\begin{equation}
{\rm{Im}}\Big ( {\rm{e}}^{-\sqrt{-1} \hat{\Theta}({\bf{E}})}\big (\omega_{0} \otimes \mathbbm{1}_{{\bf{E}}} - \frac{1}{2\pi}F_{\nabla} \big )^{3} \Big ) = 0.
\end{equation}
In order to construct an explicit example which illustrates the above construction, consider $m_{1} = \pi$. From this, we take two different integer solutions of the equation
\begin{equation}
\label{eqlagrangianpi}
\arctan \bigg ( \frac{s_{1}}{2}\bigg) + \arctan \bigg ( \frac{s_{2}}{2}\bigg) + \arctan \bigg(\frac{s_{1} + s_{2}}{4} \bigg) = \pi,
\end{equation}
and consider the associated line bundles ${\bf{F}},{\bf{G}} \in \mathcal{L}_{\pi}(\omega_{0})$. It is worth pointing out that the solutions of Eq. (\ref{eqlagrangianpi}) satisfy
\begin{equation}
s_{1}s_{2}=12, \ \ s_{1} > 0.
\end{equation}
Thus, we can take, for instance,
\begin{equation}
{\bf{F}} = \mathscr{O}_{\alpha_{1}}(2) \otimes  \mathscr{O}_{\alpha_{2}}(6) \ \ \ {\text{and}} \ \ \ {\bf{G}} = \mathscr{O}_{\alpha_{1}}(3) \otimes  \mathscr{O}_{\alpha_{2}}(4).
\end{equation}
From above, we have 
\begin{equation}
\Lambda_{\omega_{0}}(\chi_{{\bf{F}}}) = 6 \ \ {\text{and}} \ \ \Lambda_{\omega_{0}}(\chi_{{\bf{G}}}) = \frac{21}{4}.
\end{equation}
Therefore, if we consider
\begin{equation}
m_{2} = 6 \ \ \ {\text{and}} \ \ \ m_{3} =  \frac{21}{4},
\end{equation}
it follows that $m_{1} = \pi$, $m_{2} = 6$, and  $m_{3} =  \frac{21}{4}$, satisfy the desired properties. In this case, we can define
\begin{equation}
{\bf{E}}:= \underbrace{\Big ( \mathscr{O}_{\alpha_{1}}(2) \otimes  \mathscr{O}_{\alpha_{2}}(6)\Big)}_{{\bf{F}}} \oplus \underbrace{\Big ( \mathscr{O}_{\alpha_{1}}(3) \otimes  \mathscr{O}_{\alpha_{2}}(4)\Big)}_{{\bf{G}}}.
\end{equation}
By following Remark \ref{HYMinstantonlinebundle} and the previous ideas, we have that there exists a Hermitian structure ${\bf{h}}$ on ${\bf{E}}$, such that the associated Chern connection $\nabla$ is given by
\begin{equation}
\nabla|_{U^{-}(B)} = {\rm{d}} + \begin{pmatrix} A_{{\bf{F}}} & 0 \\
 0 & A_{{\bf{G}}} \end{pmatrix},
\end{equation}
such that 
\begin{itemize}
\item $A_{{\bf{F}}}=  - \partial \log \Bigg [ \bigg ( 1 + \displaystyle{\sum_{i = 1}^{2}|z_{i}|^{2}} \bigg )^{2} \bigg (1 + |z_{3}|^{2} + \bigg | \det \begin{pmatrix}
z_{1} & 1  \\                  
z_{2}  & z_{3} 
 \end{pmatrix} \bigg |^{2} \bigg )^{6} \Bigg ],$
 \item $A_{{\bf{G}}}=  - \partial \log \Bigg [ \bigg ( 1 + \displaystyle{\sum_{i = 1}^{2}|z_{i}|^{2}} \bigg )^{3} \bigg (1 + |z_{3}|^{2} + \bigg | \det \begin{pmatrix}
z_{1} & 1  \\                  
z_{2}  & z_{3} 
 \end{pmatrix} \bigg |^{2} \bigg )^{4} \Bigg ].$
\end{itemize}
By construction, we have $\Theta_{\omega_{0}}(\chi_{{\bf{F}}}) = \Theta_{\omega_{0}}(\chi_{{\bf{G}}}) =\pi$, thus
\begin{equation}
{\rm{Im}}\Big ( {\rm{e}}^{-\sqrt{-1} \hat{\Theta}({\bf{E}})}\big (\omega_{0} \otimes \mathbbm{1}_{{\bf{E}}} - \frac{1}{2\pi}F_{\nabla} \big )^{3} \Big ) = 0,
\end{equation}
such that 
\begin{equation}
\hat{\Theta}({\bf{E}}) = {\rm{Arg}}\int_{{\mathbbm{P}}(T_{{\mathbbm{P}^{2}}})}{\rm{tr}}\Big (\omega_{0} \otimes \mathbbm{1}_{{\bf{E}}} - \frac{1}{2\pi}F_{\nabla}\Big)^{3} = \pi \ (\text{mod} \ 2\pi).
\end{equation}
Therefore, $\nabla$ defines a dHYM connection on ${\bf{E}}$ which is not an HYM instanton. In conclusion, we have the following result.

\begin{lemma}
\label{L3}
There exists a Hermitian homlomorphic vector bundle $({\bf{E}},{\bf{h}}) \to ({\mathbbm{P}}(T_{{\mathbbm{P}^{2}}}),\omega_{0})$, with $\rank({\bf{E}}) > 1$, such that the associated Chern connection $\nabla$ satisfies 
\begin{equation}
\begin{cases}\sqrt{-1}\Lambda_{\omega_{0}}(F_{\nabla}) \neq c\mathbbm{1}_{{\bf{E}}}, \\ {\rm{Im}}\Big ({\rm{e}}^{-\sqrt{-1}\hat{\Theta}({\bf{E}})}\Big (\omega_{0} \otimes \mathbbm{1}_{{\bf{E}}} - \frac{1}{2\pi}F_{\nabla}\Big )^{3}\Big ) =  0,\end{cases}
\end{equation}
for all $c \in \mathbbm{R}$, where $\hat{\Theta}({\bf{E}}) = {\rm{Arg}}\int_{{\mathbbm{P}}(T_{{\mathbbm{P}^{2}}})}{\rm{tr}}\Big (\omega_{0} \otimes \mathbbm{1}_{{\bf{E}}} - \frac{1}{2\pi}F_{\nabla}\Big)^{3} \ (\text{mod} \ 2\pi)$. In particular, we have that ${\bf{E}}$ is slope-unstable.
\end{lemma}

\begin{remark}
It is worth mentioning that the subset $\mathcal{D}_{\pi}(\omega_{0})$ used in the previous construction is the finite subset of ${\rm{Pic}}({\mathbbm{P}}(T_{{\mathbbm{P}^{2}}}))$ of line bundles (up to isomorphism) of the form
\begin{equation}
{\bf{L}} = \mathscr{O}_{\alpha_{1}}(s) \otimes  \mathscr{O}_{\alpha_{2}}(\textstyle{\frac{12}{s}}), \ \ s \in \mathbbm{N}, \ s | 12.
\end{equation}
If we define
\begin{equation}
{\bf{E}}:= \bigoplus_{s \in \mathbbm{N},  s | 12} \Big (\mathscr{O}_{\alpha_{1}}(s) \otimes  \mathscr{O}_{\alpha_{2}}(\textstyle{\frac{12}{s}}) \Big ) \ \ {\text{or}} \ \ \displaystyle{\bigoplus_{s \in \mathbbm{N},  s | 12} \Big (\mathscr{O}_{\alpha_{1}}({\textstyle{\frac{12}{s}}}) \otimes  \mathscr{O}_{\alpha_{2}}(s) \Big )},
\end{equation}
one can check that ${\bf{E}} \to {\mathbbm{P}}(T_{{\mathbbm{P}^{2}}})$ is a rank 6 (slope unstable) holomorphic vector bundle. Proceeding similarly as in the last example, one can construct a dHYM connection on ${\bf{E}}$ which is not dHYM instanton.
\end{remark}

\begin{proof}(Theorem \ref{thmArestate}) The proof follows from Lemma \ref{L1}, Lemma \ref{L2}, and Lemma \ref{L3}.

\end{proof}

\section{Proofs of Theorem B and Theorem C}
\label{Proof_theo_B_C}
Let $X_{P}$ be a rational homogeneous variety and let $\omega_{0} \in \Omega^{2}(X_{P})$ be a $G$-invariant K\"{a}hler form. In this setting, we have the following.
\begin{theorem}[Theorem \ref{TheoB}]
Given ${\bf{E}}, {\bf{F}} \in {\rm{Pic}}(X_{P})$, if 
\begin{equation}
{\rm{Im}}\bigg ( \frac{Z_{[\omega_{0}]}({\bf{E}})}{Z_{[\omega_{0}]}({\bf{F}})} \bigg ) = 0,
\end{equation}
then ${\bf{E}} \oplus {\bf{F}}$ admits a dHYM instanton.
\end{theorem}

\begin{proof}
Given ${\bf{E}} \in {\rm{Pic}}(X_{P})$, it follows that 
\begin{equation}
Z_{[\omega_{0}]}({\bf{E}}) = -\frac{(-\sqrt{-1})^{n}}{n!}\int_{X_{P}}\big([\omega_{0}] + \sqrt{-1}c_{1}({\bf{E}})\big)^{n}.
\end{equation}
Therefore, $\forall {\bf{E}}, {\bf{F}} \in {\rm{Pic}}(X_{P})$, we have 
\begin{equation}
{\rm{Im}}\bigg ( \frac{Z_{[\omega_{0}]}({\bf{E}})}{Z_{[\omega_{0}]}({\bf{F}})} \bigg ) = {\rm{Im}}\Bigg (\frac{\int_{X_{P}}\big([\omega_{0}] + \sqrt{-1}c_{1}({\bf{E}})\big)^{n}}{\int_{X_{P}}\big([\omega_{0}] + \sqrt{-1}c_{1}({\bf{F}})\big)^{n}} \Bigg ).
\end{equation}
Considering the $G$-invariant representatives $\chi_{{\bf{E}}} \in c_{1}({\bf{E}})$ and $\chi_{{\bf{F}}} \in c_{1}({\bf{F}})$, it follows that 
$$\textstyle{{\rm{Im}}\Big ( \frac{Z_{[\omega_{0}]}({\bf{E}})}{Z_{[\omega_{0}]}({\bf{F}})} \Big) = 0 \iff \underbrace{{\rm{Arg}} \Big ( \int_{X_{P}}(\omega_{0} + \sqrt{-1}\chi_{{\bf{E}}})^{n}\Big )}_{\Theta_{\omega_{0}}(\chi_{{\bf{E}}})} - \underbrace{{\rm{Arg}} \Big ( \int_{X_{P}}(\omega_{0} + \sqrt{-1}\chi_{{\bf{F}}})^{n}\Big )}_{\Theta_{\omega_{0}}(\chi_{{\bf{F}}})} \in 2\pi \mathbbm{Z},}$$
see Remark \ref{dHYMisHYM}. By taking the Hermitian structures ${\bf{h}}_{{{\bf{E}}}}$ on  ${\bf{E}}$ and ${\bf{h}}_{{{\bf{F}}}}$ on  ${\bf{F}}$, such that the curvatures of the associated Chern connections $\nabla^{{\bf{E}}}$ and $\nabla^{{\bf{F}}}$ satisfy
\begin{equation}
\frac{\sqrt{-1}}{2\pi}F_{\nabla^{{\bf{E}}}} = \chi_{{\bf{E}}} \ \ \ \ \ {\text{and}} \ \ \ \ \ \frac{\sqrt{-1}}{2\pi}F_{\nabla^{{\bf{F}}}} = \chi_{{\bf{F}}},
\end{equation}
it follows from Eq. (\ref{phaseanglehigherbundle}) that
\begin{equation}
{\rm{Im}}\bigg ( \frac{Z_{[\omega_{0}]}({\bf{E}})}{Z_{[\omega_{0}]}({\bf{F}})} \bigg ) = 0 \iff \hat{\Theta}({\bf{E}}) = \hat{\Theta}({\bf{F}}) \ ({\rm{mod}} \ 2\pi).
\end{equation}
Considering the Hermitian structure ${\bf{h}}$ on ${\bf{E}} \oplus {\bf{F}}$ induced by ${\bf{h}}_{{{\bf{E}}}}$ and ${\bf{h}}_{{{\bf{F}}}}$, it follows that the curvature of the associated Chern connection $\nabla = \nabla^{{\bf{E}}} \oplus \nabla^{{\bf{F}}}$ is given by
\begin{equation}
\frac{\sqrt{-1}}{2\pi}F_{\nabla} = \begin{pmatrix} \chi_{{\bf{E}}} & 0 \\
 0 & \chi_{{\bf{F}}} \end{pmatrix}.
\end{equation}
From above, we obtain
\begin{equation}
\int_{X_{P}}{\rm{tr}}\Big (\omega_{0} \otimes \mathbbm{1}_{{\bf{E}}\oplus {\bf{F}}} - \frac{1}{2\pi}F_{\nabla}\Big)^{n}= \int_{X_{P}}\Big (\omega_{0}  + \sqrt{-1} \chi_{{\bf{E}}}\Big)^{n} +  \int_{X_{P}}\Big (\omega_{0}  + \sqrt{-1} \chi_{{\bf{F}}}\Big)^{n}.
\end{equation}
Thus, we conclude that 
\begin{equation}
\label{equivalentphase}
{\rm{Im}}\bigg ( \frac{Z_{[\omega_{0}]}({\bf{E}})}{Z_{[\omega_{0}]}({\bf{F}})} \bigg ) = 0 \iff \begin{cases} \hat{\Theta}({\bf{E}} \oplus {\bf{F}}) = \hat{\Theta}({\bf{E}}) \ ({\rm{mod}} \ 2\pi) \\ \hat{\Theta}({\bf{E}} \oplus {\bf{F}}) = \hat{\Theta}({\bf{F}}) \ ({\rm{mod}} \ 2\pi)\end{cases}.
\end{equation}
Since
\begin{center}
${\rm{e}}^{-\sqrt{-1} \hat{\Theta}({\bf{E}} \oplus {\bf{F}})}\Big (\omega_{0} \otimes \mathbbm{1}_{{\bf{E}}\oplus {\bf{F}}} - \frac{1}{2\pi}F_{\nabla} \Big )^{n}  = {\rm{e}}^{-\sqrt{-1} \hat{\Theta}({\bf{E}} \oplus {\bf{F}})}\begin{pmatrix} (\omega_{0}  + \sqrt{-1} \chi_{{\bf{E}}})^{n} & 0 \\
 0 & (\omega_{0}  + \sqrt{-1} \chi_{{\bf{F}}})^{n}\end{pmatrix}$,
\end{center}
from Eq. (\ref{equivalentphase}) and from Theorem \ref{ProofTheo1} (see Remark \ref{dHYMisHYM}), we conclude that 
\begin{equation}
{\rm{Im}}\bigg ( \frac{Z_{[\omega_{0}]}({\bf{E}})}{Z_{[\omega_{0}]}({\bf{F}})} \bigg ) = 0 \Rightarrow {\rm{Im}} \Bigg ( {\rm{e}}^{-\sqrt{-1} \hat{\Theta}({\bf{E}} \oplus {\bf{F}})}\Big (\omega_{0} \otimes \mathbbm{1}_{{\bf{E}}\oplus {\bf{F}}} - \frac{1}{2\pi}F_{\nabla} \Big )^{n}\Bigg) = 0,
\end{equation}
i.e., $\nabla = \nabla^{{\bf{E}}} \oplus \nabla^{{\bf{F}}}$ is a dHYM connection on ${\bf{E}} \oplus {\bf{F}}$.
\end{proof}

As a consequence of the above result and the ideas introduced in Section \ref{SecProoftheoA}, we have the following theorem.

\begin{theorem}[Theorem \ref{theoC}]
Under the hypotheses of Theorem \ref{TheoA}, for every integer $r > 1$, there exists a Hermitian holomorphic vector bundle $({\bf{E}},{\bf{h}}) \to ({\mathbbm{P}}(T_{{\mathbbm{P}^{2}}}),\omega_{0})$, such that $\rank({\bf{E}}) = r$, and the following hold:
\begin{enumerate}
\item ${\bf{E}}$ is slope-unstable and $h^{0}({\mathbbm{P}}(T_{{\mathbbm{P}^{2}}}),{\rm{End}}({\bf{E}})) > 1$;
\item considering $\varphi({\bf{E}}) = {\rm{Arg}}\big(Z_{[\omega_{0}]}({\bf{E}})\big )\ ({\rm{mod}} \ 2\pi)$, we have
\begin{equation}
{\rm{Im}}\Big ({\rm{e}}^{-\sqrt{-1}\varphi({\bf{E}})}Z_{\omega_{0}}({\bf{E}},\nabla) \Big) = 0,
\end{equation}
\end{enumerate}
where $\nabla$ is the Chern connection associated with ${\bf{h}}$.
\end{theorem}

\begin{proof}
Given a Hermitian holomorphic vector bundle $({\bf{E}},{\bf{h}}) \to ({\mathbbm{P}}(T_{{\mathbbm{P}^{2}}}),\omega_{0})$, denoting by $\nabla$ the associated Chern connection, we have that
\begin{equation}
\widehat{{\rm{ch}}({\bf{E}},\nabla)} = \exp \bigg (\frac{\sqrt{-1}}{2\pi}F_{\nabla} \bigg )  = \widehat{{\rm{ch}}_{0}({\bf{E}},\nabla)} + \widehat{{\rm{ch}}_{1}({\bf{E}},\nabla)} + \widehat{{\rm{ch}}_{2}({\bf{E}},\nabla)} + \widehat{{\rm{ch}}_{3}({\bf{E}},\nabla)}, 
\end{equation}
such that 
\begin{equation}
\widehat{{\rm{ch}}_{k}({\bf{E}},\nabla)} = \frac{1}{k!} \Big ( \frac{\sqrt{-1}}{2\pi}F_{\nabla}\Big )^{k}  = \frac{1}{k!}  \underbrace{ \Big ( \frac{\sqrt{-1}}{2\pi}F_{\nabla}\Big ) \wedge \cdots \wedge \Big ( \frac{\sqrt{-1}}{2\pi}F_{\nabla}\Big )}_{k-{\text{times}}},
\end{equation}
for all $k = 0,1,2,3$. From above, considering the ${\rm{End}}({\bf{E}})$-valued $(3,3)$-form
\begin{equation}
Z_{\omega_{0}}({\bf{E}},\nabla) := -\sum_{j=0}^{3}\frac{(-\sqrt{-1})^{j}}{j!}\omega_{0}^{j} \wedge \widehat{{\rm{ch}}_{3-j}({\bf{E}},\nabla)},
\end{equation}
it follows that 
\begin{equation}
Z_{\omega_{0}}({\bf{E}},\nabla) = -\sum_{j=0}^{3}\frac{(-\sqrt{-1})^{j}}{j!} \Bigg ( \frac{1}{(3-j)!} \omega_{0}^{j} \wedge \Big ( \frac{\sqrt{-1}}{2\pi}F_{\nabla}\Big )^{3-j} \Bigg ).
\end{equation}
Since $\frac{1}{j!(3-j)!} = \frac{1}{3!} \binom{3}{j}$, $\forall j = 0,\ldots,3$, we obtain
\begin{equation}
\label{binomial1}
Z_{\omega_{0}}({\bf{E}},\nabla) = -\sum_{j=0}^{3}\frac{(-\sqrt{-1})^{j}}{3!} \Bigg ( \binom{3}{j} \omega_{0}^{j} \wedge \Big ( \frac{\sqrt{-1}}{2\pi}F_{\nabla}\Big )^{3-j} \Bigg ).
\end{equation}
Now we observe that 
\begin{equation}
\label{binomial2}
\Big (  \omega \otimes \mathbbm{1}_{{\bf{E}}} - \frac{1}{2\pi}F_{\nabla} \Big )^{3} = \sum_{j = 0}^{3}(-1)^{3-j}\binom{3}{j} \omega_{0}^{j} \wedge \Big ( \frac{1}{2\pi}F_{\nabla}\Big )^{3-j}.
\end{equation}
Therefore, replacing $(-1)^{j} = (-1)^{3}(-1)^{3-j}$, $j = 0,\ldots,3$, in Eq. (\ref{binomial1}), it follows from Eq. (\ref{binomial2}) that
\begin{equation}
Z_{\omega_{0}}({\bf{E}},\nabla) = -\frac{(-\sqrt{-1})^{3}}{3!}\Big ( \omega \otimes \mathbbm{1}_{{\bf{E}}}  - \frac{1}{2\pi}F_{\nabla} \Big )^{3}.
\end{equation}
In particular, if ${\bf{E}} = {\bf{L}}_{1} \oplus \cdots \oplus {\bf{L}}_{r}$, such that ${\bf{L}}_{\ell} \in {\rm{Pic}}({\mathbbm{P}}(T_{{\mathbbm{P}^{2}}}))$, for every $\ell = 1,\ldots,r$, by taking the ${\rm{SU}}(3)$-invariant representative $\chi_{{\bf{L}}_{\ell}} \in c_{1}({\bf{L}}_{\ell})$, $\forall \ell = 1,\ldots,r$, and choosing Hermitian structures ${\bf{h}}_{\ell}$, such that $\frac{\sqrt{-1}}{2\pi}F_{\nabla^{(\ell)}} = \chi_{{\bf{L}}_{\ell}}$, where $\nabla^{(\ell)}$ is the associated Chern connection of ${\bf{h}}_{\ell}$, for all $\ell = 1,\ldots,r$, we have an induced Hermitian structure ${\bf{h}}$ on ${\bf{E}}$, such that $\nabla = \nabla^{(1)} \oplus \cdots \oplus \nabla^{(r)}$ is the associated Chern connection of ${\bf{h}}$. Thus, we obtain
\begin{equation}
\label{matrixform}
Z_{\omega_{0}}({\bf{E}},\nabla) = - \frac{(-\sqrt{-1})^{3}}{3!}\begin{pmatrix} (\omega_{0} + \sqrt{-1}\chi_{{\bf{L}}_{1}})^{3} & \cdots & 0 \\
 \vdots & \ddots & \vdots \\
 0 & \cdots & (\omega_{0} + \sqrt{-1}\chi_{{\bf{L}}_{r}})^{3}\end{pmatrix}.
\end{equation}
On the other hand, we have
\begin{equation}
Z_{[\omega_{0}]}({\bf{E}}) = -\int_{{\mathbbm{P}}(T_{{\mathbbm{P}^{2}}})}{\rm{e}}^{-\sqrt{-1}[\omega_{0}]}{\rm{ch}}({\bf{E}}) = -\sum_{\ell = 1}^{r} \int_{{\mathbbm{P}}(T_{{\mathbbm{P}^{2}}})}{\rm{e}}^{-\sqrt{-1}[\omega_{0}]}{\rm{ch}}({\bf{L}}_{\ell}).
\end{equation}
Since 
\begin{equation}
 \underbrace{-\int_{{\mathbbm{P}}(T_{{\mathbbm{P}^{2}}})}{\rm{e}}^{-\sqrt{-1}[\omega_{0}]}{\rm{ch}}({\bf{L}}_{\ell})}_{Z_{[\omega_{0}]}({\bf{L}}_{\ell})} = -\frac{(-\sqrt{-1})^{3}}{3!}\int_{{\mathbbm{P}}(T_{{\mathbbm{P}^{2}}})}\big([\omega_{0}] + \sqrt{-1}c_{1}({\bf{L}}_{\ell})\big)^{3},
\end{equation}
if we suppose that 
\begin{equation}
{\rm{Im}}\bigg ( \frac{Z_{[\omega_{0}]}({\bf{L}}_{\ell})}{Z_{[\omega_{0}]}({\bf{L}}_{\ell+1})} \bigg ) = 0 \iff \hat{\Theta}({\bf{L}}_{\ell}) = \hat{\Theta}({\bf{L}}_{\ell+1}) \ ({\rm{mod}} \ 2\pi).
\end{equation}
for all $\ell = 1,\ldots,r-1$, it follows that 
\begin{equation}
\varphi({\bf{E}}) = {\rm{Arg}} \Big ( Z_{[\omega_{0}]}({\bf{E}})\Big ) = \Big [\hat{\Theta}({\bf{L}}_{\ell}) + \frac{3\pi}{2} \Big]\ ({\rm{mod}} \ 2\pi),
\end{equation}
for every $\ell = 1,\ldots,r$, notice that $-(-\sqrt{-1})^{3} = {\rm{e}}^{\frac{3\pi}{2}\sqrt{-1}}$. Since $\hat{\Theta}({\bf{L}}_{\ell}) = \Theta_{\omega_{0}}(\chi_{{\bf{L}}_{\ell}}) \ ({\rm{mod}} \ 2\pi)$, for all $\ell = 1,\ldots,r$, see for instance Eq. (\ref{phaselinebundle}), it follows from Eq. (\ref{matrixform}) that 
\begin{equation}
{\rm{Im}} \Big ({\rm{e}}^{-\sqrt{-1}\varphi({\bf{E}})}Z_{\omega_{0}}({\bf{E}},\nabla) \Big ) = 0 \iff  {\rm{Im}} \Big ({\rm{e}}^{-\sqrt{-1}\Theta_{\omega_{0}}(\chi_{{\bf{L}}_{\ell}})}(\omega_{0} + \sqrt{-1}\chi_{{\bf{L}}_{\ell}})^{3} \Big ) = 0, 
\end{equation}
$\forall \ell =1,\ldots,r$. From above, in order to conclude the proof, one can consider, for instance, ${\bf{L}}_{1}, \ldots,{\bf{L}}_{r} \in {\rm{Pic}}({\mathbbm{P}}(T_{{\mathbbm{P}^{2}}}))$, such that 
\begin{equation}
{\bf{L}}_{1} = \mathscr{O}_{\alpha_{1}}(2) \otimes  \mathscr{O}_{\alpha_{2}}(6) \ \ \ {\text{and}} \ \ \ {\bf{L}}_{\ell} = \mathscr{O}_{\alpha_{1}}(3) \otimes  \mathscr{O}_{\alpha_{2}}(4), \ \forall \ell = 2, \ldots,r.
\end{equation}
Defining ${\bf{E}} := {\bf{L}}_{1} \oplus \cdots \oplus {\bf{L}}_{r}$, we notice that 
\begin{enumerate}
\item[(A)] $\mu_{[\omega_{0}]}({\bf{L}}_{1}) \neq \mu_{[\omega_{0}]}({\bf{L}}_{\ell})$, $\forall \ell = 2,\ldots,r$,
\item[(B)] $\hat{\Theta}({\bf{L}}_{\ell}) = \pi \ ({\rm{mod}} \ 2\pi)$, $\forall \ell = 1,\ldots,r$,
\item[(C)]${\rm{End}}({\bf{E}}) = {\rm{End}}({\bf{L}}_{1}) \oplus \cdots \oplus {\rm{End}}({\bf{L}}_{r})$,
\end{enumerate}
for item (A) and item (B), see Section \ref{Unstableconstruction}. From item (A) we have that ${\bf{E}}$ is slope-unstable, from item (B) and the construction presented above we have that there exists a Hermitian structure ${\bf{h}}$ on ${\bf{E}}$, such that the associated Chern connection $\nabla$ is a solution of the equation
\begin{equation}
{\rm{Im}} \Big ({\rm{e}}^{-\sqrt{-1}\varphi({\bf{E}})}Z_{\omega_{0}}({\bf{E}},\nabla) \Big ) = 0.
\end{equation}
From item (C), it follows that 
\begin{equation}
h^{0}({\mathbbm{P}}(T_{{\mathbbm{P}^{2}}}),{\rm{End}}({\bf{E}})) = \dim \big ( H^{0}({\mathbbm{P}}(T_{{\mathbbm{P}^{2}}}),{\rm{End}}({\bf{E}}))\big ) > 1,
\end{equation}
which concludes the proof.
\end{proof}

\bibliographystyle{alpha}
\bibliography{biblio}

\end{document}